\documentclass[a4paper, 11pt]{amsart}
\usepackage{mathrsfs}
\usepackage{amsfonts}
\usepackage{amssymb}
\usepackage{amsxtra}
\usepackage{dsfont}
\usepackage{color}
\usepackage{graphicx}
\usepackage[english,polish]{babel}
\usepackage{enumerate}
\usepackage{mathtools}

\usepackage{newtxmath}
\usepackage{newtxtext}

\usepackage[margin=2.5cm, centering]{geometry}
\usepackage[colorlinks,citecolor=blue,urlcolor=blue,bookmarks=true]{hyperref}
\hypersetup{
pdfpagemode=UseNone,
pdfstartview=FitH,
pdfdisplaydoctitle=true,
pdfborder={0 0 0}, 
pdftitle={Asymptotics of orthogonal polynomials with slowly oscillating recurrence coefficients},
pdfauthor={Grzegorz Świderski and Bartosz Trojan},
pdflang=en-US
}

\usepackage[utf8]{inputenc}

\newcommand{\CC}{\mathbb{C}}
\newcommand{\ZZ}{\mathbb{Z}}
\newcommand{\sS}{\mathbb{S}}
\newcommand{\NN}{\mathbb{N}}
\newcommand{\RR}{\mathbb{R}}

\newcommand{\calC}{\mathcal{C}}

\newcommand{\calB}{\mathcal{B}}

\newcommand{\calA}{\mathcal{A}}

\newcommand{\calX}{\mathcal{X}}

\newcommand{\calD}{\mathcal{D}}

\newcommand{\scrD}{\mathscr{D}}

\newcommand{\vphi}{\varphi}

\newcommand{\Id}{\operatorname{Id}}

\newcommand{\pl}[1]{\foreignlanguage{polish}{#1}}

\newcommand{\abs}[1]{\lvert {#1} \rvert}
\newcommand{\sprod}[2]{\langle {#1}, {#2} \rangle}
\newcommand{\tr}{\operatorname{tr}}

\newcommand{\GL}{\operatorname{GL}}
\newcommand{\Mat}{\operatorname{Mat}}

\newcommand{\diag}{\operatorname{diag}}
\newcommand{\discr}{\operatorname{discr}}
\newcommand{\sym}{\operatorname{sym}}

\newcommand{\per}{{\rm per}}
\newcommand{\sigmaEss}[1]{\sigma_{\mathrm{ess}}(#1)}

\newcommand{\ud}{{\: \rm d}}

\newtheorem{theorem}{Theorem}
\newtheorem{proposition}{Proposition}
\newtheorem{lemma}{Lemma}
\newtheorem{corollary}{Corollary}
\newtheorem{claim}{Claim}

\newtheorem*{theorem*}{Theorem}

\theoremstyle{definition}

\newtheorem{remark}{Remark}
\newtheorem{definition}{Definition}

\numberwithin{equation}{section}

\theoremstyle{plain}
\newcounter{thm}

\newtheorem{main_theorem}[thm]{Theorem}

\newcounter{cor}

\title[Asymptotics of orthogonal polynomials]
{Asymptotics of orthogonal polynomials with slowly oscillating recurrence coefficients}
\date{\today}

\author{Grzegorz Świderski}

\address{
	\pl{
		Grzegorz Świderski\\
		The Institute of Mathematics\\
   	 	Polish Academy of Science\\
    	ul. \'Sniadeckich 8\\
    	00-696 Warszawa\\
    	Poland}
}
\email{gswiderski@impan.pl}

\author{Bartosz Trojan}
\address{
	\pl{
	Bartosz Trojan\\
	The Institute of Mathematics\\
	Polish Academy of Science\\
	ul. \'Sniadeckich 8\\
	00-696 Warszawa\\
	Poland}
}
\email{btrojan@impan.pl}



\begin{document}
\selectlanguage{english}

\begin{abstract}
We study solutions of three-term recurrence relations whose $N$-step transfer matrices belong to
the uniform Stolz class. In particular, we derive the first order of their uniform asymptotics. 
For orthonormal polynomials we show more. Namely, we find the constructive formula for 
the density of their orthogonality measure in terms of Tur\'an determinants and we determine their
exact asymptotic behavior. We treat both bounded and unbounded cases in a uniform manner. 
\end{abstract}

\maketitle

\section{Introduction}
Let us consider two sequences $a = (a_n \colon n \in \NN_0)$ and $b = (b_n \colon n \in \NN_0)$ of positive 
and real numbers, respectively. Then one defines the symmetric tridiagonal matrix by the formula
\begin{equation*}
  \calA =
  \begin{pmatrix}
     b_0 & a_0 & 0   & 0      &\ldots \\
     a_0 & b_1 & a_1 & 0       & \ldots \\
     0   & a_1 & b_2 & a_2     & \ldots \\
     0   & 0   & a_2 & b_3   &  \\
     \vdots & \vdots & \vdots  &  & \ddots
  \end{pmatrix}.
\end{equation*}
The action of $\calA$ on \emph{any} sequence is defined by the formal matrix multiplication. 
Let $A$ be the minimal operator associated with $\calA$, that is the closure in $\ell^2$ of
the restriction of $\calA$ to the set of sequences having finite support. Here $\ell^2$
denotes the Hilbert space of square-summable sequences endowed with the scalar product
\[
	\sprod{x}{y}_{\ell^2} = \sum_{n=0}^\infty x_n \overline{y_n}.
\]
The operator $A$ is called a~\emph{Jacobi matrix.} Since both deficiency indices of $A$ are equal 
(see, e.g. \cite[Corollary 6.7]{Schmudgen2017}), it always has a self-adjoint extension which is unique, 
if the Carleman condition
\[
	\sum_{n=0}^\infty \frac{1}{a_n} = \infty
\]
is satisfied (see, e.g. \cite[Corollary 6.19]{Schmudgen2017}).

A generalized eigenvector $(u_n : n \in \NN_0)$ associated with $x \in \RR$ is a non-zero 
sequence satisfying the recurrence relation
\[
	a_{n-1} u_{n-1} + b_n u_n + a_n u_{n+1} = x u_n, \qquad n \geq 1,
\]
which can be written as
\[
	\begin{pmatrix}
		u_n \\
		u_{n+1}
	\end{pmatrix}
	=
	B_n(x)
	\begin{pmatrix}
		u_{n-1} \\
		u_n
	\end{pmatrix}
\]
where
\[
	B_n(x) = 
	\begin{pmatrix}
		0 & 1 \\
		-\frac{a_{n-1}}{a_n} & 	\frac{x-b_n}{a_n}
	\end{pmatrix}.
\]
For each $\alpha \in \RR^2 \setminus \{0\}$ there is the unique generalized eigenvector $(u_n)$ such that
$(u_0, u_1) = \alpha$. 

By $(p_n(x) : n \in \NN_0)$ we denote the generalized eigenvectors associated with $x$ and satisfying the initial
condition
\[
	p_0(x) = 1, \qquad p_1(x) = \frac{x-b_0}{a_0}.
\]
If $\tilde{A}$ is a self-adjoint extension of $A$, then $E_{\tilde{A}}$, the spectral resolution of the identity for 
$\tilde{A}$, gives rise to the Borel measure on $\RR$,
\begin{equation} \label{eq:156}
	\tilde{\mu}(B) = \sprod{E_{\tilde{A}}(B) \delta_0}{\delta_0}_{\ell^2}
\end{equation}
where $\delta_0$ is the sequence having $1$ on the $0$th position and $0$ elsewhere. The polynomials
$(p_n : n \in \NN_0)$ form an orthonormal basis in $L^2(\RR, \tilde{\mu})$, which is the Hilbert space of
square-integrable functions equipped with the scalar product
\[
	\sprod{f}{g}_{L^2(\tilde{\mu})} = \int_\RR f(x) \overline{g(x)} \: \tilde{\mu}({\rm d} x).
\]
Each self-adjoint extension of $A$ leads to a different measure. It turns out that, if the self-adjoint extension is not
unique then the measures \eqref{eq:156} are purely discrete (see, e.g. \cite[Theorem 7.7]{Schmudgen2017}). Moreover, 
there are measures $\nu$ such that $(p_n)$ are orthonormal in $L^2(\RR, \nu)$, but do not correspond 
to any self-adjoint extension of $A$ (see, e.g. \cite[Theorem 1]{Berg1981}). 

The asymptotic behavior of generalized eigenvectors entail properties of the operator $A$. 
Specifically, the operator $A$ is self-adjoint if and only if there is a generalized eigenvector 
which is not square-summable (see, e.g. \cite[Theorem 6.16]{Schmudgen2017}). Moreover, in view of the
subordination theory (see, e.g. \cite[Theorem 2.1]{Clark1996}), for any interval $I \subset \RR$, if $A$ is self-adjoint
and for every pair of generalized eigenvectors $(u_n)$ and $(v_n)$ associated with $x \in I$, 
\[
	\limsup_{n \to \infty} \frac{\sum_{j = 0}^n u_j^2}{\sum_{j = 0}^n v_j^2} < \infty,
\]
then the operator $A$ is absolutely continuous on $I$, and $I$ is in the spectrum of $A$.

The aim of this paper is threefold: determine the asymptotic behavior of the generalized eigenvectors, find the density
of an orthogonality measure, and determine the asymptotic behavior of the orthogonal polynomials. About the
recurrence coefficients we assume that they are slowly oscillating. To be more precise, given a compact set
$K \subset \RR$ and $r \in \NN$, we say that the uniformly bounded sequence of continuous mappings 
$A_n : K \rightarrow \GL(2, \RR)$
belongs to $\calD_r \big( K, \GL(2, \RR) \big)$, if for each $j \in \{1, \ldots, r\}$,
\[
	\sum_{n \geq 1} \sup_{x \in K} \big\|\Delta^j A_n(x) \big\|^\frac{r}{j} < \infty
\]
where
\begin{align*}
	\Delta^0 A_n &= A_n, \\
	\Delta^j A_n &= \Delta^{j-1} A_{n+1} - \Delta^{j-1} A_n, \qquad n \geq 1.
\end{align*}
Our first result discuss the asymptotic behavior of generalized eigenvectors.
\begin{main_theorem} 
	\label{thm:A}
	Let $N$ and $r$ be positive integers and $i \in \{0, \ldots, N-1\}$. We set 
	\[
		X_n = \prod_{j = n}^{n+N-1} B_j.
	\]
	Let $K$ be a compact subset of
	\footnote{A discriminant of a $2 \times 2$ matrix $X$ is $\discr X = (\tr X)^2 - 4 \det X$.}
	\[
		\Lambda = 
		\Big\{x \in \RR : \lim_{n \to \infty} 
		\discr\big(X_{nN+i}(x)\big) \text{ exists and is negative}
		\Big\}.
	\]
	Assume that
	\[
		\lim_{n \to \infty} \frac{a_{(n+1)N+i-1}}{a_{nN+i-1}} = 1.
	\]
	If $(X_{nN+i} : n \in \NN)$ belongs to $\calD_r \big( K, \GL(2, \RR) \big)$, then there is a constant $c>1$
	such that for every generalized eigenvector $(u_n : n \in \NN_0)$ associated with $x \in K$, and all $n \geq 1$,
	\[
		c^{-1} \big( u_0^2 + u_1^2 \big) 
		\leq 
		a_{nN+i-1} \big( u_{nN+i-1}^2 + u_{nN+i}^2 \big)
		\leq
		c \big( u_0^2 + u_1^2 \big).
	\]
\end{main_theorem}
Theorem \ref{thm:A} is a consequence of Corollary~\ref{cor:8} whose proof relies on studying 
the asymptotic behavior of $N$-shifted generalized Tur\'an determinants defined for a generalized
eigenvector $(u_n : n \in \NN_0)$ by
\[
	S_n = a_{n+N-1} \big( u_n u_{n+N-1} - u_{n-1} u_{n+N} \big).
\]
To make the most from the slowly oscillating nature of the sequence $(X_{nN+i} : n \in \NN)$, we extend and simplify the
method of iterated diagonalization introduced in \cite{Stolz1994}, see Section \ref{sec:1} for details.

Classical Tur\'an determinants, that is the expression $p_n^2(x) - p_{n+1}(x) p_{n-1}(x)$, were used for the first time 
in \cite{Turan1950} as a tool in studying the zeros of Legendre polynomials. In fact, it was observed that they
are non-negative on the support of the orthogonality measure. Later, in \cite{Nevai1979}, the convergence of the Tur\'an
determinants was investigated for orthogonal polynomials with recurrence coefficients satisfying
\begin{equation} 
	\label{eq:132}
	\sum_{n=0}^\infty |a_n - 1| < \infty, \qquad\text{and}\qquad
	\sum_{n=0}^\infty |b_n| < \infty.
\end{equation}
It was proven that the limit is related to the density of the orthogonality measure. In \cite{Nevai1983}, the
assumption \eqref{eq:132} was replaced by the condition of bounded variation imposed on $(a_n)$ and $(b_n)$, which
was extended in \cite{GeronimoVanAssche1991} to their $N$-periodic perturbations. See also \cite{Nevai1987} for the subsequent developments. Recently, in \cite{PeriodicII},
the first author obtained analogous result for unbounded sequences $(a_n)$ and $(b_n)$.

Meantime, in \cite{Totik1985}, it was observed that the method successful in proving convergence of Tur\'an determinants
can also be used in studying the asymptotic behavior of the orthogonal polynomials $(p_n)$ with
recurrence coefficients having bounded variation. Their $N$-periodic perturbations were investigated in 
\cite{GeronimoVanAssche1991}. Lastly, the case of unbounded sequences was recently considered in
\cite{AptekarevGeronimo2016}.

The initial study of the slowly oscillating sequences goes back to \cite{Stolz1994} where the absolute continuity
was obtained for the case $a_n \equiv 1$. The pointwise asymptotic formula for generalized eigenvector was proven in
\cite{Moszynski2006}. Finally, in \cite{Moszynski2009}, the absolutely continuity was studied both in bounded and
unbounded cases.

The following theorem provides a formula for the absolutely continuous part of an orthogonality
measure. It is a consequence of Corollary~\ref{cor:9}.
\begin{main_theorem}
	\label{thm:B}
	Let $N$ and $r$ be positive integers and $i \in \{0, 1, \ldots, N-1\}$. Let $K$ be a compact interval contained in
	\[
		\Lambda = 
		\Big\{x \in \RR : \lim_{n \to \infty} 
		\discr\big(X_{nN+i}(x)\big) \text{ exists and is negative}
		\Big\}.
	\]
	Assume that
	\[	
		\lim_{n \to \infty} \frac{a_{(n+1)N+i-1}}{a_{nN+i-1}} = 1.
	\]
	If $(X_{nN+i} : n \in \NN)$ belongs to $\calD_r \big(K, \GL(2, \RR))$, then there is a positive function
	$g: K \rightarrow \RR$, such that
	\[
		\lim_{\stackrel{n \to \infty}{n \equiv i \bmod N}} 
		\sup_{x \in K}
		\Big|
		a_{n+N-1} \big|p_{n}(x) p_{n+N-1}(x) - p_{n-1}(x) p_{n+N}(x) \big| - g(x)
		\Big|
		=0.
	\]
	Moreover, there is a probability measure $\nu$ such that $(p_n : n \in \NN_0)$ are orthonormal in $L^2(\RR, \nu)$,
	which is absolutely continuous on $K$ with the density
	\[
		\nu'(x) = \frac{\sqrt{-h(x)}}{2\pi g(x)}, \qquad x \in K
	\]
	where
	\[
		h(x) = \lim_{n \to \infty} \discr \big( X_{nN+i}(x) \big), \qquad x \in K.
	\]
\end{main_theorem}
The main result of this article is the following theorem which gives the uniform asymptotic of the orthogonal
polynomials $(p_n : n \in \NN_0)$, see Corollary~\ref{cor:7} for the proof.
\begin{main_theorem}
	\label{thm:C}
	Let $N$ and $r$ be positive integers and $i \in \{0, 1, \ldots, N-1\}$. Suppose that $K$ is a compact
	interval contained in
	\[
		\Lambda = 
		\Big\{x \in \RR : \lim_{n \to \infty} 
		\discr\big(X_{nN+i}(x)\big) \text{ exists and is negative}
		\Big\}.
	\]
	Assume that
	\[
		\lim_{n \to \infty} \frac{a_{(n+1)N+i-1}}{a_{nN+i-1}} = 1,
	\]
	and
	\[
		(X_{nN+i} : n \in \NN) \in \calD_r \big(K, \GL(2, \RR) \big).
	\]
	Suppose that $\calX: K \rightarrow \GL(2, \RR)$ is the limit of $(X_{nN+i} : n \in \NN_0)$. Then there are $M \geq 1$ and a continuous real-valued function $\eta$ such that for all $n \geq M$,
	\[
		\lim_{n \to \infty} \sup_{x \in K} 
		\bigg|
		 \sqrt{a_{nN+i-1}} p_{nN+i}(x) -
		 \sqrt{\frac{2 \big|[\calX(x)]_{2,1}\big|}{\pi \nu'(x) \sqrt{-h(x)}}}
		 \sin \Big( \sum_{j=M+1}^{n} \theta_j(x) + \eta(x) \Big)
		\bigg| = 0.
	\]
	where $\nu$ and $h$ are as in Theorem \ref{thm:B}, and $\theta_j: K \rightarrow (0, \pi)$ are some
	continuous functions satisfying
	\[
		\lim_{j \to \infty} \sup_{x \in K} 
		\bigg| \theta_j(x) - \arccos \Big( \tfrac{1}{2} \tr \calX(x) \Big) \bigg| = 0.
	\]
\end{main_theorem}
Theorems \ref{thm:B} and \ref{thm:C} are consequences of Theorems \ref{thm:6} and \ref{thm:3}. 
Their proofs are based on a proper truncation and periodization of the recurrence coefficients. This method allows us
to show the convergence by applying the corresponding results for eventually periodic sequences.

Let us close the introduction by giving an example of new sequences covered by Theorems \ref{thm:A},
\ref{thm:B} and \ref{thm:C}. For $\gamma \in (0, 1)$, we set
\[
	a_n \equiv 1, \qquad b_n = \frac{\cos(n^\gamma)}{\log(n+2)}.
\]
Then the hypotheses of our theorems are satisfied for $N=1$ and any $r > 1/(1-\gamma)$. 
Nevertheless, $(b_n : n \in \NN_0)$ is neither of bounded variation nor belongs to any $\ell^p$. 
Hence, the results of \cite{Lukic2011} and \cite{Nevai1983} cannot be applied. In Section
\ref{sec:2}, we present more applications and special cases covered by this paper.

The structure of the paper is as follows: In Section \ref{sec:1}, we develop the main tool, Theorem \ref{thm:4},
which allows us to construct in Section \ref{sec:3} uniform variant of iterated diagonalization, see Theorem \ref{thm:5}.
The convergence of generalized $N$-shifted Tur\'an determinants are proven in Section \ref{sec:4}, see Theorem \ref{thm:1}.
In Section \ref{sec:5}, we describe the truncation method that seems to be of independent interest. In particular,
we show the formula for density of orthonormalizing measure, see Theorem \ref{thm:6}. The next section is devoted 
to prove asymptotic of the orthogonal polynomials, see Theorem \ref{thm:3}. Finally, Section \ref{sec:2} contains 
some applications.

\subsection{Notation}
By $\NN$ we denote the set of positive integers and $\NN_0 = \NN \cup \{0\}$. For any compact set $K$, by $o_K(1)$
we denote the class of functions $f_n : K \rightarrow \RR$ such that $\lim_{n \to \infty} f_n(x) = 0$
uniformly with respect to $x \in K$. Moreover, by $c$ we denote generic positive constants whose value may change
from line to line.

\section{Stolz class}
\label{sec:1}
In this section we define a proper class of slowly oscillating sequences which is motivated by \cite{Stolz1994}.

For a sequence $(x_n : n \in \NN)$ of elements from a normed space $X$ and $j \in \NN_0$, we set
\[
	\begin{aligned}
	\Delta^0 x_n &= x_n \\
	\Delta^j x_n &= \Delta^{j-1} x_{n+1} - \Delta^{j-1} x_n, \qquad \text{for } n \geq 1. \\
	\end{aligned}
\]
We say that a bounded sequence $(x_n)$ belongs to $\calD_{r, s}(X)$ for some $r \in \NN$ and $s \in \{0, \ldots, r-1\}$,
if for each $j \in \{1, \ldots, r-s\}$,
\[
	\sum_{n \in \NN} \big\| \Delta^j x_n \big\|^{\frac{r}{j+s}} < \infty.
\]
Moreover, $(x_n) \in \calD_{r,s }^0(X)$, if $(x_n) \in \calD_{r,s}(X)$ and
\[
	\sum_{n \in \NN} \|x_n\|^{\frac{r}{s}} < \infty.
\]
Let us observe that for $s \in \{0, 1, \ldots, r-2 \}$
\begin{equation} \label{eq:4}
	\Delta^1 f \in \calD_{r, s+1}^0(X) \quad \text{provided} \quad 
	f \in \calD_{r, s}(X).
\end{equation}
Moreover,
\begin{equation} \label{eq:8}
	\calD_{r, s}(X) \subset \calD_{r, 0}(X).
\end{equation}
Indeed, for $j \in \{1, \ldots, r - s\}$, we have
\[
	\big\| \Delta^j x_n \big\|^{\frac{r}{j}} \leq c \big\| \Delta^j x_n \big\|^{\frac{r}{s+j}}.
\]
If $j \in \{r-s+1, \ldots, r\}$, then
\begin{align*}
	\big\| \Delta^j x_n \big\|^{\frac{r}{j}} 
	&\leq 
	c \sum_{k = 0}^{s+j-r} \big\| \Delta^{r-s} x_{n+k} \big\|^{\frac{r}{j}} \\
	&\leq
	c \sum_{k = 0}^{s+j-r} \big\| \Delta^{r-s} x_{n+k} \big\|.
\end{align*}
To simplify the notation, if $X$ is the real line with an Euclidean norm we shortly write
$\calD_{r, s} = \calD_{r, s}(X)$. Given a compact set $K \subset \CC$ and a vector space $R$, by $\calD_{r, s}(K, R)$ we
denote the case when $X$ is the space of all continuous mappings from $K$ to $R$ equipped with a supremum norm.

The following lemma is well-known and its proof is straightforward.
\begin{lemma}
	\label{lem:1}
	For any two sequences $(x_n)$ and $(y_n)$, we set $z_n = x_n y_n$. Then for each $j \in \NN$,
	\[
		\Delta^j z_n = \sum_{k = 0}^j {j \choose k} \Delta^{j-k}x_n \cdot
		\Delta^k y_{n + j - k}.
	\]
\end{lemma}

The following shows that $\calD_{r,s}(X)$ is an algebra and $\calD^0_{r,s}(X)$ is its ideal.
\begin{corollary}
	\label{cor:1}
	Let $r \in \NN$ and $s \in \{0, \ldots, r-1\}$. 
	\begin{enumerate}[(i)]
		\item \label{en:1}
		If $(x_n) \in \calD_{r, s}(X)$ and $(y_n) \in \calD_{r, s}(X)$ then $(x_n y_n) \in \calD_{r, s}(X)$. 
	\item \label{en:2}
		If $(x_n) \in \calD_{r, 0}(X)$, and $(y_n) \in \calD^0_{r, s}(X)$ then $(x_n y_n) \in \calD_{r, s}^0(X)$.
	\end{enumerate}
\end{corollary}
\begin{proof}
	In view of Lemma \ref{lem:1}, it is enough to estimate
	\[
		\sum_{n \in \NN} 
		\Big\| \Delta^{j-k}x_n \cdot \Delta^ky_{n+j-k} \Big\|^{\frac{r}{s+j}}
	\]
	for any $j \in \{1, \ldots, r-s\}$ and $k \in \{0, \ldots, j\}$. If $k = 0$, then
	\[
		\sum_{n \in \NN} \big\| \Delta^j x_n \cdot y_{n+j} \big\|^{\frac{r}{s + j}} 
		\leq
		\Big(
		\sup_{n \in \NN} \big\|y_n\big\|^{\frac{r}{s+j}} \Big)
		\sum_{n \in \NN} \big\| \Delta^j x_n \big\|^{\frac{r}{s+j}}.
	\]
	Similarly, for $k = j$ we have
	\[
		\sum_{n \in \NN} \big\| x_n \cdot \Delta^j y_n \big\|^{\frac{r}{s + j}} 
		\leq
		\Big(
		\sup_{n \in \NN} \big\|x_n\big\|^{\frac{r}{s+j}} \Big)
		\sum_{n \in \NN} \big\| \Delta^j y_n \big\|^{\frac{r}{s+j}}.
	\]
	Finally, if $k \in \{1, \ldots, j-1\}$ then by H\"older's inequality
	\[
		\sum_{n \in \NN} \Big\| \Delta^{j-k} x_n \cdot \Delta^k y_{n+j-k} \Big\|^{\frac{r}{s+j}}
		\leq
		\bigg(\sum_{n \in \NN} \Big\| \Delta^{j-k} x_n \Big\|^{\frac{r}{j-k}} \bigg)^{\frac{j-k}{s+j}}
		\bigg(\sum_{n \in \NN} \Big\| \Delta^k y_{n+j-k} \Big\|^{\frac{r}{s+k}}\bigg)^{\frac{s+k}{s+j}}.
	\]
	Since $\calD_{r, s}(X) \subset \calD_{r, 0}(X)$, we conclude the proof of \eqref{en:1}. The reasoning for 
	\eqref{en:2} is similar.
\end{proof}

The following result shows that $\calD_{r, s}(K, \RR)$ is closed under the composition with smooth
maps.
\begin{lemma}
	\label{lem:2}
	Fix $r \in \NN$, $s \in \{0, \ldots, r-1\}$ and a compact set $K \subset \RR$. Let 
	$(f_n : n \in \NN) \in \calD_{r, s}(K, \RR)$ be a sequence of real functions on $K$ with values in $I \subseteq \RR$
	and let $F \in \calC^{r-s}(I, \RR)$. Then $(F \circ f_n : n \in \NN) \in \calD_{r, s}(K, \RR)$.
\end{lemma}
\begin{proof}
	We first show that for each $k \in \{1, \ldots, r- s\}$, $\Delta^k (F \circ f_n)$ is a finite linear combination of
	terms of a form
	\[
        \int_{A_{1, n}}^{B_{1, n}} \cdots \int_{A_{m, n}}^{B_{m, n}}
        F^{(m)}(t_1 + \cdots + t_m) {\: \rm d}t_1\cdots {\: \rm d}t_m
    \]
	with $1 \leq m \leq k$, where for each $j \in \{1, \ldots, m\}$, $A_j = (A_{j, n} : n \in \NN)$ and 
	$B_j = (B_{j, n} : n \in \NN)$ are functions on $K$ such that
    \[
		A_j, B_j \in \calD_{r, s}(K, \RR), \qquad\text{and}\qquad
		B_j - A_j \in \calD_{r, s_j}^0(K, \RR),
	\]
	with
    \[
        \sum_{j = 1}^m s_j = s + k.
    \]
	The reasoning is by induction over $k \in \{1, \ldots, r-s\}$. For $k = 1$, we have
	\[
		\Delta^1 (F \circ f_n)(x) = F \big(f_{n+1}(x)\big) - F \big(f_n(x) \big) = 
		\int^{f_{n+1}(x)}_{f_n(x)} F'(t) {\: \rm d}t
	\]
	where by \eqref{eq:4}
	\[
		(f_{n+1} - f_n : n \in \NN) \in \calD_{r, s+1}^0(K, \RR).
	\]
	Since
	\begin{equation}
		\label{eq:16}
		\Delta^{k+1} (F \circ f_n) = \Delta^k (F \circ f_{n+1}) - \Delta^k (F \circ f_n),
	\end{equation}
	by the inductive hypothesis, the right-hand side of \eqref{eq:16} is a linear combination of terms having a form
	\[
		\bigg(
		\int_{A_{1, n+1}}^{B_{1, n+1}} \cdots \int_{A_{m, n+1}}^{B_{m, n+1}}
		-
		\int_{A_{1, n}}^{B_{1, n}} \cdots \int_{A_{m, n}}^{B_{m, n}}
		\bigg)
        F^{(m)}(t_1 + \cdots + t_m) {\: \rm d}t_1\cdots {\: \rm d}t_m.
	\]
	We write
	\begin{align*}
		&
		\int_{A_{1, n+1}}^{B_{1, n+1}} \cdots \int_{A_{m, n+1}}^{B_{m, n+1}}
        -
        \int_{A_{1, n}}^{B_{1, n}} \cdots \int_{A_{m, n}}^{B_{m, n}} \\
		&\qquad\qquad=
		\sum_{j = 1}^m
		\int_{A_{1, n+1}}^{B_{1, n+1}} \cdots \int_{A_{j-1, n+1}}^{B_{j-1, n+1}}
		\bigg(
		\int_{A_{j, n+1}}^{B_{j, n+1}} - \int_{A_{j, n}}^{B_{j, n}}
		\bigg) 
		\int_{A_{j+1, n}}^{B_{j+1, n}}
		\cdots
		\int_{A_{m, n}}^{B_{m, n}}.
	\end{align*}
	Notice that
	\[
		\bigg(
		\int_{A_{n+1}}^{B_{n+1}} - \int_{A_n}^{B_n}
		\bigg)
		G(t) {\: \rm d} t
		=
		\int_{A_{n+1}}^{B_{n+1}} \int_{A_n - A_{n+1}}^0 G'(s + t) {\: \rm d} s {\: \rm d} t
		+
		\int_{B_n}^{B_{n+1} + A_n - A_{n+1}} G(t) {\: \rm d} t.
	\]
	Moreover, if $(A_n : n \in \NN), (B_n : n \in \NN) \in \calD_{r, s}(K, \RR)$, and 
	$(B_n - A_n : n \in \NN) \in \calD_{r, s'}^0(K, \RR)$, then by \eqref{eq:4} and \eqref{eq:8}, 
	\[
		(A_{n+1} - A_n : n \in \NN) \in \calD_{r, s+1}^0(K, \RR) \subset \calD_{r, 1}^0(K, \RR), 
	\]
	and 
	\[
		(B_{n+1} - A_{n+1} - B_n + A_n : n \in \NN) \in \calD_{r, s'+1}^0(K, \RR).
	\]
	Therefore, 
	\begin{align*}
		&\int_{A_{1, n+1}}^{B_{1, n+1}} \cdots \int_{A_{j-1, n+1}}^{B_{j-1, n+1}}
        \bigg(
        \int_{A_{j, n+1}}^{B_{j, n+1}} - \int_{A_{j, n}}^{B_{j, n}}
        \bigg)
        \int_{A_{j+1, n}}^{B_{j+1, n}}
        \cdots
        \int_{A_{m, n}}^{B_{m, n}}
        F^{(m)}(t_1+\cdots+t_m)
        {\: \rm d}t_1\cdots {\: \rm d}t_m\\
		&\quad
		=
		\int_{A_{1, n+1}}^{B_{1, n+1}} \cdots 
        \int_{A_{j, n+1}}^{B_{j, n+1}}
        \int_{A_{j+1, n}}^{B_{j+1, n}}
        \cdots
        \int_{A_{m, n}}^{B_{m, n}}
		\int_{A_{j, n} - A_{j, n+1}}^0
        F^{(m+1)}(t_1+\cdots+t_{m+1}) 
        {\: \rm d}t_1\cdots{\: \rm d}t_{m+1} \\
		&\quad\phantom{=}+
		\int_{A_{1, n+1}}^{B_{1, n+1}} \cdots 
		\int_{A_{j-1, n+1}}^{B_{j-1, n+1}}
        \int_{B_{j, n}}^{B_{j, n+1}+A_{j, n} - A_{j, n+1}} 
        \int_{A_{j+1, n}}^{B_{j+1, n}}
        \cdots
        \int_{A_{m, n}}^{B_{m, n}}
        F^{(m)}(t_1+\cdots+t_m)
        {\: \rm d}t_1\cdots {\: \rm d}t_m,
	\end{align*}
	proving the claim.

	Next, since $F \in \calC^{r-s}(I, \RR)$, each term of a form
	\[
		\int_{A_{1, n}}^{B_{1, n}} \cdots \int_{A_{m, n}}^{B_{m, n}}
        F^{(m)}(t_1 + \cdots + t_m) {\: \rm d}t_1\cdots {\: \rm d}t_m,
	\]
	for $1 \leq m \leq k \leq r-s$, is bounded by a constant multiple of
	\[
		\prod_{j=1}^m \abs{B_{j, n} - A_{j, n}}.
	\]
	Because $(B_{j, n} - A_{j, n} : n \in \NN) \in \calD_{r, s_j}^0(K, \RR)$ with
	\[
		\sum_{j=1}^m s_j = s + k,
	\]
	by H\"older's inequality, we obtain
	\begin{align*}
		&\sum_{n \in \NN} \sup_K \bigg|
        \int_{A_{1, n}}^{B_{1, n}} \cdots \int_{A_{m, n}}^{B_{m, n}}
        F^{(m)}(t_1 + \cdots + t_m) {\: \rm d}t_1\cdots {\: \rm d}t_m
        \bigg|^{\frac{r}{s+k}}\\
		&\qquad\qquad\leq
		c^{\frac{r}{s+k}}
		\sum_{n \in \NN}
		\prod_{j = 1}^m
		\sup_K \big|B_{j, n} - A_{j, n}\big|^{\frac{r}{s+k}} \\
		&\qquad\qquad\leq
		c^{\frac{r}{s+k}}
		\prod_{j = 1}^m
		\bigg(
		\sum_{n \in \NN} \sup_K \big|B_{j, n} - A_{j, n} \big|^\frac{r}{s_j}
		\bigg)^{\frac{s_j}{s+k}} < \infty.
	\end{align*}
	In view of the claim, we conclude that
	\[
		\sum_{n \in \NN} \sup_K \big|\Delta^k (F \circ f_n) \big|^{\frac{r}{s+k}} < \infty,
	\]
	for each $1 \leq k \leq r-s$, thus $(F \circ f_n ) \in \calD_{r, s}(K, \RR)$.
\end{proof}
\begin{corollary}
	\label{cor:2}
	Fix $r \in \NN$, $s \in \{0, \ldots, r-1\}$ and a compact set $K \subset \RR$. If $(x_n : n \in \NN) \in 
	\calD_{r, s}(K, \CC)$, and 
	\[
		\inf_{n \in \NN} \inf_{x \in K} \abs{x_n(x)} > 0,
	\]
	then $(x_n^{-1} : n \in \NN) \in \calD_{r, s}(K, \CC)$.
\end{corollary}

Our main tool will be the following theorem based on \cite[Theorem 4]{Stolz1994}. 
\begin{theorem}
	\label{thm:4}
	Fix two integers $r > 2$, $s \in \{0, \ldots, r-2\}$ and a compact set $K \subset \RR$. Suppose that
	$(\lambda_n : n \in \NN)$ is a uniform Cauchy sequence belonging to $\calD_{r, s}(K, \CC)$ and such that for all
	$x \in K$ and $n \in \NN$,
	\[
		\Im \lambda_n(x) \geq \delta > 0.
	\]
	Let $(X_n : n \in \NN) \in \calD_{r, s}\big(K, \GL(2, \CC)\big)$ be such that
	\begin{equation}
		\label{eq:2}
		\sup_{x \in K} \sup_{n \in \NN} \big(\|X_n(x)\| + \|X_n^{-1}(x)\|\big) < \infty,
	\end{equation}
	and
	\begin{equation}
		\label{eq:3}
		\begin{pmatrix}
			0 & 1\\
			1 & 0
		\end{pmatrix}
		X_n
		\begin{pmatrix}
			0 & 1 \\
			1 & 0
		\end{pmatrix}
		=
		\overline{X_n}
		\qquad \text{or} \qquad
		X_n
		\begin{pmatrix}
			0 & 1\\
			1 & 0
		\end{pmatrix}
		=
		\overline{X_n}.
	\end{equation}
	Then there are sequences $(\gamma_n : n \in \NN) \in \calD_{r, s+1}(K, \CC)$ and $(Y_n : n \in \NN) \in 
	\calD_{r, s+1}\big(K, \GL(2, \CC)\big)$ satisfying
	\[
		\begin{pmatrix}
		\lambda_n & 0 \\
		0 & \overline{\lambda_n}
		\end{pmatrix}
		X_n^{-1} X_{n-1} = 
		Y_n 
		\begin{pmatrix}
		\gamma_n & 0 \\
		0 & \overline{\gamma_n}
		\end{pmatrix}
		Y_n^{-1},
	\]
	where $(\gamma_n : n \in \NN)$ is a uniform Cauchy sequence such that
	\begin{equation} \label{eq:1}
		\lim_{n \to \infty} \sup_{x \in K} \big|\gamma_n(x) - \lambda_n(x) \big| = 0.
	\end{equation}
	Moreover,
	\begin{equation}
		\label{eq:5}
		\lim_{n \to \infty} \sup_{x \in K} \big\| Y_n(x) - \Id \big\| = 0,
	\end{equation}
	and
	\[
		\begin{pmatrix}
		0 & 1 \\
		1 & 0
		\end{pmatrix}
		Y_n
		\begin{pmatrix}
		0 & 1 \\
		1 & 0
		\end{pmatrix}
		=
		\overline{Y_n}.
	\]
\end{theorem}
\begin{proof}
	Let
	\[
		D_n = 
		\begin{pmatrix}
		\lambda_n & 0 \\
		0 & \overline{\lambda_n}
		\end{pmatrix}.
	\]
	We set
	\[
		W_n = D_n X_n^{-1} X_{n-1} = D_n \big(\Id - X_n^{-1} \Delta X_{n-1}\big).
	\]
	By \eqref{eq:2}, we have
	\[
		\sup_{K} \big\|W_n - D_n \big\| = \sup_{K} \big\|D_n X_n^{-1} \Delta X_{n-1} \big\| 
		\leq 
		c
		\sup_K \big\| \Delta X_{n-1} \big\|.
	\]
	Since $(X_n) \in \calD_{r, s}\big(K, \GL(2, \CC)\big)$,
	\[
		\lim_{n \to \infty} \sup_K \| \Delta X_n \|= 0,
	\]
	thus
	\[
		\lim_{n \to \infty} \sup_K \big\| W_n - D_n \big\|= 0.
	\]
	By \eqref{eq:3}, 
	\begin{equation} \label{eq:13}
        \begin{pmatrix}
        0 & 1 \\
        1 & 0
        \end{pmatrix}
        W_n
        \begin{pmatrix}
        0 & 1 \\
        1 & 0
        \end{pmatrix}
        =
        \overline{D_n}
        \begin{pmatrix}
        0 & 1 \\
        1 & 0
        \end{pmatrix}
        X_n^{-1}
        X_{n-1}
        \begin{pmatrix}
        0 & 1 \\
        1 & 0
        \end{pmatrix}
        =
        \overline{W_n}.
	\end{equation}
    In particular,
    \begin{equation} \label{eq:6}
    	\tr W_n = \overline{\tr W_n} \quad \text{and} \quad 
    	\det W_n = \overline{\det W_n}.
    \end{equation}
	Consequently, $W_n$ has eigenvalues $\gamma_n$ and $\overline{\gamma_n}$ such that
	\[
		\lim_{n \to \infty} \sup_K \big|\gamma_n - \lambda_n \big| = 0,
	\]
	and hence $(\gamma_n : n \in \NN)$ is a uniform Cauchy sequence satisfying \eqref{eq:1}. Setting
	\begin{equation} \label{eq:9}
		X_n = \begin{pmatrix}
		x_{11}^{(n)} & x_{12}^{(n)} \\
		x_{21}^{(n)} & x_{22}^{(n)}
		\end{pmatrix},
		\qquad\text{and}\qquad
		W_n =
		\begin{pmatrix}
		w_{11}^{(n)} & w_{12}^{(n)} \\
		w_{21}^{(n)} & w_{22}^{(n)}
		\end{pmatrix}, 
	\end{equation}
	we obtain
	\[
		W_n
		=
		\frac{1}{\det X_n}
		\begin{pmatrix}
		\lambda_n \big(x_{11}^{(n-1)} x_{22}^{(n)} - x_{21}^{(n-1)} x_{12}^{(n)}\big) &
		\lambda_n \big(x_{12}^{(n-1)} x_{22}^{(n)} - x_{22}^{(n-1)} x_{12}^{(n)}\big) \\
		\overline{\lambda_n} \big( x_{21}^{(n-1)} x_{11}^{(n)} - x_{11}^{(n-1)} x_{21}^{(n)}\big) & 
		\overline{\lambda_n} \big( x_{22}^{(n-1)} x_{11}^{(n)} - x_{12}^{(n-1)} x_{21}^{(n)}\big)
		\end{pmatrix}.
	\]
	By \eqref{eq:2}, \eqref{eq:6} and Corollaries \ref{cor:1}(i) and \ref{cor:2}, we have
	\[
		\bigg(\frac{1}{\det X_n} \bigg) \in \calD_{r, s}(K, \RR).
	\]
	Therefore, by Corollary \ref{cor:1}(i) we get
	\[
		(W_n : n \in \NN) \in \calD_{r, s}\big(K, \GL(2, \CC) \big).
	\]
	Next, we write
	\begin{align*}
		w_{12}^{(n)} 
		&= 
		\frac{\lambda_n}{\det X_n} \big(x_{12}^{(n-1)} x_{22}^{(n)} - x_{22}^{(n-1)} x_{12}^{(n)}\big) \\
		&=
		\frac{\lambda_n}{\det X_n} \Big(\big(x_{22}^{(n)} - x_{22}^{(n-1)}\big) x_{12}^{(n)} 
		- \big(x_{12}^{(n)}-x_{12}^{(n-1)}\big) x_{22}^{(n)}\Big),
	\end{align*}
	thus, by \eqref{eq:4}, \eqref{eq:8} and Corollary \ref{cor:1}(ii),
	\[
		\big(w_{12}^{(n)} : n \in \NN \big), \big(w_{21}^{(n)} : n \in \NN \big)
		\in \calD_{r, s+1}^0(K, \CC).
	\]
	By \eqref{eq:6} and \eqref{eq:9}, we have
	\[
		\gamma_n = \frac{w_{11}^{(n)}+w_{22}^{(n)}}{2} 
		+ \frac{i}{2} \sqrt{\big|\big(w_{11}^{(n)}-w_{22}^{(n)}\big)^2 + 4 w_{12}^{(n)} w_{21}^{(n)}\big|},
	\]
	and since for all $n$ sufficiently large
	\[
		\big| w_{11}^{(n)}-w_{22}^{(n)} \big| \geq 2 \Im \lambda_n - \big|w_{11}^{(n)}-\lambda_n\big|
        - \big|w_{22}^{(n)} - \overline{\lambda_n} \big| \geq \delta,
	\]
	by Lemma \ref{lem:2}, we have $(\gamma_n) \in \calD_{r, s}(K, \CC)$. Next, using \eqref{eq:13}, we
	verify that if
	\[
		W_n 
		\begin{pmatrix}
			1 \\
			v_n
		\end{pmatrix}
		=
		\gamma_n
		\begin{pmatrix}
			1 \\
			v_n
		\end{pmatrix}, 
		\quad \text{then} \quad
		W_n 
		\begin{pmatrix}
			\overline{v_n} \\
			1
		\end{pmatrix}
		=
		\overline{\gamma_n}
		\begin{pmatrix}
			\overline{v_n} \\
			1
		\end{pmatrix}.
	\]
	Hence, the matrix $Y_n$ has a form
	\[
		Y_n = 
		\begin{pmatrix}
			1 & \overline{v_n} \\
			v_n & 1
		\end{pmatrix}
	\]
	provided that the equation
	\[
		\gamma_n 		
		\begin{pmatrix}
			1 \\
			v_n
		\end{pmatrix}
		=
		W_n 
		\begin{pmatrix}
			1 \\
			v_n
		\end{pmatrix}
	\]
	has a solution. By \eqref{eq:9}, it is equivalent to the system
	\[
		\begin{cases}
			\gamma_n &= w_{11}^{(n)} + w_{12}^{(n)} v_n \\
			\gamma_n v_n &= w_{21}^{(n)} + w_{22}^{(n)} v_n.
		\end{cases}
	\]
	By inserting the first equation into the second, we arrive at the quadratic equation for $v_n$,
	which has a solution
	\[
		v_n = 
		\frac{w_{22}^{(n)}-w_{11}^{(n)} 
		+ i \sqrt{\big|\big(w_{22}^{(n)} - w_{11}^{(n)}\big)^2 + 4 w_{12}^{(n)} w_{21}^{(n)}\big|}}
		{2 w_{12}^{(n)}}.
	\]
	Next, by \eqref{eq:13} we obtain $w_{11}^{(n)} = \overline{w_{22}^{(n)}}$. Thus, by 
	direct computations, one gets
	\[
		v_n = -i \frac{w_{21}^{(n)}}{\Im \big(w_{11}^{(n)} + \gamma_n \big)}.
	\]
	Since for all $x \in K$ and $n$ sufficiently large,
	\[
		c \geq \big|\Im\big(w_{11}^{(n)}(x) + \gamma_n(x)\big)\big| \geq \delta,
	\]
	thus by Corollaries \ref{cor:2} and \ref{cor:1}(ii), we conclude that $(Y_n)$ belongs to
	$\calD_{r, s+1}\big(K, \GL(2, \CC)\big)$. Because
	\[
		\lim_{n \to \infty} \sup_K{\abs{v_n}} = 0,
	\]
	we easily obtain \eqref{eq:5}. 
\end{proof}

\section{Iterated diagonalization}
\label{sec:3}
\subsection{Uniform diagonalization}
For a sequence of square matrices $(X_n : n \in \NN)$ and $n, m \in \NN$ we set 
\[
	\prod_{j = m}^n X_j =
	\begin{cases}
		X_n X_{n-1} \cdots X_m & \text{if } n \geq m,\\
		\Id & \text{otherwise.}
	\end{cases}
\]
\begin{definition} \label{def:1}
	Let $X = (X_n : n \in \NN)$ be a sequence of continuous mappings defined on $K$, some compact subset of $\RR$, with
	values in $\GL(2, \RR)$. Then $(X_n)$ is \emph{uniformly diagonalizable on $K$,} if there is $M>0$
	such that for all $n > m \geq M$,
	\[
		\prod_{j=m}^n X_j = Q_n \bigg( \prod_{j=m}^n D_j C_j^{-1} C_{j-1} \bigg) Q_{m-1}^{-1}
	\]
	where
	\begin{enumerate}[(a)]
	\item
	\label{def:1:eq:a}
	for every $j \geq M$ the mappings $Q_j,Q_{j-1}^{-1}, D_j, D_j^{-1}, C_{j-1}$ and $C^{-1}_j$ 
	are continuous on $K$ and such that for some $r \geq 1$,
	\[
		(C_j : j \geq M) \in \calD_{1, 0}\big(K, \GL(2, \CC) \big), \qquad\text{and}\qquad
		(D_j : j \geq M) \in \calD_{r, 0} \big( K, \diag(2, \CC) \big);
	\]
	\item \label{def:1:eq:b}
	there are non-singular matrices $C_\infty$, $Q_\infty$, and $D_\infty$ such that
	\[
		\lim_{n \to \infty} \sup_K \| C_n - C_\infty \| = 0,
		\qquad
		\lim_{n \to \infty} \sup_K \| Q_n - Q_\infty \| = 0,
		\qquad\text{and}\qquad
		\lim_{n \to \infty} \sup_K \| D_n - D_\infty \| = 0.
	\]
	\end{enumerate}
\end{definition}

The following theorem provides a sufficient condition for uniform diagonalization.
\begin{theorem}
	\label{thm:5}
	Let $(X_n : n \in \NN)$ be a sequence of continuous mappings defined on $\RR$ with values in $\GL(2, \RR)$. Let
	$K$ be a compact subset of
	\[
		\Big\{x \in \RR : \lim_{n \to \infty} \discr \big( X_n(x) \big) 
		\text{ exists and is negative} \Big\}.
	\]
	Suppose that $(X_n : n \in \NN) \in \calD_{r, 0}\big(K, \GL(2, \RR) \big)$, for some $r \geq 1$. If the sequences
	\[
		\big(\tr X_n : n \in \NN\big), \qquad\text{and}\qquad \big(\det X_n : n \in \NN \big)
	\]
	converge uniformly on $K$, then $(X_n)$ is uniformly diagonalizable on $K$ with the sequence of diagonal matrices
	$(D_j : j \geq M)$ such that
	\[
		\lim_{n, m \to \infty} \prod_{j = m}^n \Big( (\det D_j)^{-1} (\det X_j) \Big) = 1,
	\]
	uniformly on $K$. Moreover,
	\begin{equation} \label{eq:152}
		\lim_{n \to \infty} \sup_K \big\| D_n - \diag(\lambda_n, \overline{\lambda_n}) \big\| = 0,
	\end{equation}
	where
	\begin{equation} \label{eq:153}
		\lambda_n(x) = \frac{\tr X_n(x) }{2} + \frac{i}{2} \sqrt{\abs{\discr{X_n(x)}}}.
	\end{equation}
\end{theorem}
\begin{proof}
	Since the sequence $(\discr X_n)$ is uniformly convergent on $K$, there are $\delta > 0$ and $M \geq 1$
	such that for all $n \geq M$ and $x \in K$,
	\begin{equation}
		\label{eq:15}
			\discr \big( X_n(x) \big) \leq -\delta.
	\end{equation}
	Hence, for $x \in K$, the matrix $X_n(x)$ has eigenvalues $\lambda_n(x)$ and 
	$\overline{\lambda_n(x)}$ satisfying \eqref{eq:153}.
	Since $(X_n) \in \calD_{r, 0}\big(K, \GL(2, \RR) \big)$, by \eqref{eq:15} and Lemma \ref{lem:2} we obtain that
	$(\lambda_n) \in \calD_{r, 0}(K, \CC)$, and
	\[
		\Im \lambda_n(x) = \frac{1}{2} \sqrt{\abs{\discr{X_n(x)}}} \geq \frac{1}{2} \sqrt{\delta} > 0.
	\]
	We set
	\[
		C_{n, 0} = 
		\begin{pmatrix}
			1 & 1 \\
			\lambda_n & \overline{\lambda_n}
		\end{pmatrix}
		,
		\qquad\text{and}\qquad
		D_{n, 0} = 
		\begin{pmatrix}
			\lambda_n & 0 \\
			0 & \overline{\lambda_n}
		\end{pmatrix}.
	\]
	Then both $(C_{n, 0} : n \geq M)$ and $(D_{n, 0} : n \geq M)$ belong to $\calD_{r, 0}\big(K, \GL(2, \CC)\big)$ and
	\[
		X_n = C_{n, 0}D_{n, 0} C_{n, 0}^{-1}.
	\]
	By Theorem \ref{thm:4}, there are two sequence of matrices
	\[
		(C_{n, 1}: n \geq M) \in \calD_{r, 1}\big(K, \GL(2, \CC)\big) 
		\qquad \text{and} \qquad 
		(D_{n, 1}: n \geq M) \in \calD_{r, 0}\big(K, \GL(2, \CC)\big),
	\]
	such that
	\[
		D_{n, 0} C_{n, 0}^{-1} C_{n-1, 0} = C_{n, 1} D_{n, 1} C_{n, 1}^{-1},
	\]
	and
	\[
		D_{n, 1} = 
		\begin{pmatrix}
			\gamma_{n, 1} & 0 \\
			0 & \overline{\gamma_{n, 1}}
		\end{pmatrix}.
	\]
	Therefore, for $M < m < n$,
	\begin{align*}
		\prod_{j=m}^n X_j
		&= 
		\big(C_{n, 0} D_{n, 0} C_{n, 0}^{-1} \big) \cdots \big(C_{m, 0} D_{m, 0} C_{m, 0}^{-1}\big) \\
		&=
		C_{n, 0} \big(D_{n, 0} C_{n, 0}^{-1} C_{n-1, 0}\big) \cdots \big(D_{m, 0} C_{m, 0}^{-1} C_{m-1, 0}\big) 
		C_{m-1, 0}^{-1} \\
		&=
		C_{n, 0} \big(C_{n, 1} D_{n, 1} C_{n, 1}^{-1} \big)\cdots \big(C_{m, 1} D_{m, 1} C_{m, 1}^{-1}\big) 
		C_{m-1, 0}^{-1} \\
		&=
		C_{n, 0} C_{n, 1} \big(D_{n, 1} C_{n, 1}^{-1} C_{n-1, 1}\big) \cdots 
		\big(D_{m, 1} C_{m, 1}^{-1} C_{m-1,1}\big) \big(C_{m-1, 0} C_{m-1, 1}\big)^{-1}.
	\end{align*}
	By repeated application of Theorem \ref{thm:4}, for each $k \in \{2, \ldots, r-1\}$, we can find sequences
	\[
		(C_{n, k} : n \geq M) \in \calD_{r, k}\big(K, \GL(2, \CC)\big), 
		\qquad\text{and}\qquad
		(D_{n, k} : n \geq M) \in \calD_{r, 0}\big(K, \GL(2, \CC)\big),
	\]
	such that
	\begin{equation}
		\label{eq:18}
		D_{n, k-1} C_{n, k-1}^{-1} C_{n-1, k-1} = C_{n, k} D_{n, k} C_{n, k}^{-1},
	\end{equation}
	and
	\[
		D_{n, k} = 
		\begin{pmatrix}
			\gamma_{n, k} & 0 \\
			0 & \overline{\gamma_{n, k}}
		\end{pmatrix}.
	\]
	Hence,
	\[
		\prod_{j=m}^n X_j
		=
		Q_n
		\Big( \prod_{j=m}^n D_{j, r-1} C_{j, r-1}^{-1} C_{j-1, r-1} \Big)
		Q_{m-1}^{-1},
	\]
	where
	\[
		Q_n = \prod_{k=0}^{r-1} C_{n, k}
	\]
	and for every $k \in \{1, 2, \ldots, r-1 \}$
	\begin{equation}
		\label{eq:89}
		\lim_{j \to \infty} \sup_K \| C_{j, k} - \Id \| = 0.
	\end{equation}
	Furthermore, by \eqref{eq:18}, we get
	\[
		\det(D_{n, k}) = \det(D_{n, k-1}) \det \big( C_{n, k-1}^{-1} C_{n-1, k-1} \big),
	\]
	thus,
	\begin{align}
		\nonumber
		\prod_{j=m}^n 
		\det(D_{j, r-1})
		&= \prod_{j=m}^n \bigg(\abs{\lambda_{j}}^2 
		\prod_{k=0}^{r-1} 
		\det \big( C_{j,k}^{-1} C_{j-1, k} \big) \bigg) \nonumber \\
		\label{eq:62}
		&= 
		\Big( \prod_{j=m}^{n} \det X_j \Big) 
		\Big( \prod_{k=0}^{r-1} \det \big( C_{n, k}^{-1} C_{m-1, k} \big) \Big),
	\end{align}
	which together with \eqref{eq:89} completes the proof.
\end{proof}

\begin{remark} 
	\label{rem:1}
	Suppose that $(X_n : n \geq M)$ is a sequence of matrices such that
	\[
		X_n = C D C^{-1}
	\]
	where $D$ is a diagonal matrix. For $n > m > M$, by applying the reasoning presented in the 
	proof of Theorem \ref{thm:5}, we get
	\[
		\prod_{j = m}^n 
		X_n = Q_n \bigg(\prod_{j = m}^n D_j C_j^{-1} C_{j-1} \bigg) Q_{m-1}^{-1},
	\]
	where $Q_j = C$, $D_j = D$, and $C_j = \Id$.
\end{remark}

In the next two proposition, we deduce some estimates satisfied by uniformly diagonalizable sequences.
\begin{proposition} 
	\label{prop:1}
	Suppose that the sequence $(X_n : n \in \NN)$ is uniformly diagonalizable on some compact set $K$, $K \subset \RR$.
	Then there is a constant $c>0$ such that for all $m, n \geq M$, uniformly on $K$,
	\begin{equation}
		\label{eq:21}
		\bigg\| \prod_{j=m}^n D_{j} C_j^{-1} C_{j-1} \bigg\|
		\leq c \prod_{j=m}^n \| D_{j} \|,
		\qquad\text{and}\qquad
		\bigg\| \Big( \prod_{j=m}^n D_{j} C_j^{-1} C_{j-1} \Big)^{-1} \bigg\|
   		\leq c \prod_{j=m}^n \| D_{j} \|^{-1},
	\end{equation}
	and 
	\begin{equation} 
		\label{eq:22}
		\Big\| \prod_{j=m}^n D_{j} C_j^{-1} C_{j-1} - \prod_{j=m}^n D_{j} \Big\| 
		\leq c \Big( \prod_{j=m}^n \| D_{j} \| \Big) \cdot \sum_{j=m}^n \sup_K \| \Delta C_{j-1} \|.
	\end{equation}
\end{proposition}
\begin{proof}
	Let us show the first inequality in \eqref{eq:21}. We have
	\[
		\Big\| \prod_{j=m}^n D_{j} C_{j}^{-1} C_{j-1} \Big\|
		\leq
		\Big( \prod_{j=m}^n \| D_{j} \| \Big)
		\cdot
		\Big(
		\prod_{j = m}^n \big(1 + \| C_{j}^{-1} \| \cdot \| \Delta C_{j-1} \| \big) \Big).
	\]
	Since $(C_{j}^{-1} : j \geq M)$ is uniformly bounded and
	\[
		\sum_{j=M}^\infty \sup_K \| \Delta C_{j} \| < \infty,
	\]
	we easily get
	\begin{equation}
		\label{eq:23}
		\Big\| \prod_{j=m}^n D_{j, r-1} C_{j}^{-1} C_{j-1} \Big\| \leq  
		c \prod_{j=m}^n \| D_{j} \|,
	\end{equation}
	for some $c > 0$. Similarly we prove the second inequality in \eqref{eq:21}.

	Next, we write
	\[
		\bigg\|
		\prod_{j = m}^n \big(D_{j} C_{j}^{-1} C_{j-1} \big)
		-
		\prod_{j = m}^n D_{j}
		\bigg\|
		\leq
		\sum_{k = m}^n
		\bigg\|
		\bigg( \prod_{j=k}^n D_j \bigg) 
		\big(C_k^{-1} \Delta C_{k-1} \big) 
		\bigg( \prod_{j=m}^{k-1} \big(D_j C_j^{-1} C_{j-1} \big) \bigg)
		\bigg\|,
	\]
	which, by \eqref{eq:23}, is bounded by a constant multiply of
	\[
		\sum_{k=m}^n \bigg\{ 
		\Big( \prod_{j=k}^n \| D_j \| \Big) \cdot \| \Delta C_{k-1} \| \cdot
		\Big( \prod_{j=m}^{k-1} \| D_j \| \Big) \bigg\} 
		=
		\Big( \prod_{j=m}^n \| D_j \| \Big) 
		\sum_{k=m}^n \| \Delta C_{k-1} \|,
	\]
	proving \eqref{eq:22}.
\end{proof}

\begin{proposition} 
	\label{prop:7}
	Suppose that the sequence $(X_n : n \in \NN)$ is uniformly diagonalizable on some compact set $K$, $K \subset \RR$.
	Then there is a constant $c>0$ such that for all $n, m \geq M$, uniformly on $K$,
	\begin{align*}
		&
		\bigg\|\prod_{j = {m+1}}^{n+1} X_j - \prod_{j = m}^n X_j \bigg\| \\
		&\qquad\leq
		c
		\bigg(\prod_{j = m+1}^n \|D_j\|\bigg)
		\cdot
		\bigg( 
		\| Q_{n+1} - Q_{n} \| +  \| Q_{m}^{-1} - Q_{m-1}^{-1} \| + \| D_{n+1} - D_{m} \|
		+
		\sum_{j=m-1}^n \sup_K \| \Delta C_{j} \|\bigg).
	\end{align*}
\end{proposition}
\begin{proof}
	Let us observe that
	\begin{equation}
		\label{eq:24}
		\begin{aligned}
		\prod_{j = m+1}^{n+1} X_j - \prod_{j = m}^n X_j
		&=
		(Q_{n+1} - Q_n) \bigg( \prod_{j=m+1}^{n+1} D_{j} C_{j}^{-1} C_{j-1} \bigg) Q_{m}^{-1} \\
		&\phantom{=}+
		Q_n \bigg( \prod_{j=m+1}^{n+1} D_{j} C_{j}^{-1} C_{j-1} - \prod_{j=m}^n D_{j} C_{j}^{-1} C_{j-1} \bigg) Q_m^{-1}\\
		&\phantom{=}+
		Q_n \bigg( \prod_{j=m}^n D_{j} C_{j}^{-1} C_{j-1} \bigg) (Q_m^{-1} - Q_{m-1}^{-1}).
		\end{aligned}
	\end{equation}
	In view of \eqref{eq:21}, the first and the third term on the right-hand side of \eqref{eq:24} are bounded by a
	constant multiple of
	\begin{equation} 
		\label{eq:25}
		\| Q_{n+1} - Q_n \|  \Big( \prod_{j=m+1}^{n+1} \| D_{j} \| \Big) \| Q_{m}^{-1} \|,
	\end{equation}
	and
	\begin{equation} 
		\label{eq:26}
	 	\| Q_{n} \| \Big( \prod_{j=m}^n \| D_{j} \| \Big) \| Q_{m}^{-1} - Q_{m-1}^{-1} \|,
	\end{equation}
	respectively. To bound the second term in \eqref{eq:24}, let us observe that
	\begin{equation}
		\label{eq:27}
		\begin{aligned}
		\prod_{j=m+1}^{n+1} D_{j} C_{j}^{-1} C_{j-1} - \prod_{j=m}^n D_{j} C_{j}^{-1} C_{j-1} 
		&=
		\Big( \prod_{j=m+1}^{n+1} D_{j} C_{j}^{-1} C_{j-1} - \prod_{j=m+1}^{n+1} D_{j} \Big) \\
		&\phantom{=}-
		\Big( \prod_{j=m}^{n} D_{j} C_{j}^{-1} C_{j-1} - \prod_{j=m}^{n} D_{j} \Big) \\
		&\phantom{=}+
		\Big( \prod_{j=m+1}^{n+1} D_{j} - \prod_{j=m}^{n} D_{j} \Big).
		\end{aligned}
	\end{equation}
	Hence, by \eqref{eq:22}, the left-hand side of \eqref{eq:27} is bounded by a constant multiple of
	\begin{equation}
		\label{eq:28}
		\Big( \prod_{j=m+1}^{n+1} \| D_{j} \| + \prod_{j=m}^{n} \| D_{j} \| \Big) 
		\sum_{j=m-1}^n \sup_K \| \Delta C_{j} \|
		+ \| D_{n+1} - D_{m} \| \prod_{j=m+1}^{n} \| D_{j} \|.
	\end{equation}
	Since the sequences $(Q_n : n \geq M)$, $(Q_n^{-1} : n \geq M)$, and $(D_{n} : n \geq M)$ are uniformly bounded on
	$K$, putting together \eqref{eq:25}, \eqref{eq:26}, and \eqref{eq:28} we conclude the proof.
\end{proof}

\section{Generalized shifted Tur\'an determinants}
\label{sec:4}
Let $N$ be a positive integer. The generalized $N$-shifted Tur\'an determinants are defined by the formula
\[
	S_n(\alpha, x) = a_{n+N-1} \big( u_n u_{n+N-1} - u_{n-1} u_{n+N} \big)
\]
where $(u_n : n \in \NN_0)$ is a generalized eigenvector associated with $x \in \RR$ and $\alpha = (u_0, u_1) \in \RR^2
\setminus \{0\}$. Let us denote by $\sS^1$ the unit sphere in $\RR^2$. Observe that
\begin{equation}
	\label{eq:10}
	\begin{pmatrix}
		u_{n} \\
		u_{n+1}
	\end{pmatrix}
	=
	B_n(x)
	\begin{pmatrix}
		u_{n-1} \\
		u_n
	\end{pmatrix}
\end{equation}
where
\[
	B_n(x) = 
	\begin{pmatrix}
		0 & 1 \\
		-\frac{a_{n-1}}{a_n} & \frac{x - b_n}{a_n}
	\end{pmatrix}.
\]
Thus
\[
	S_n(\alpha, x) =
	a_{n+N-1} 
	\bigg\langle
	E X_n(x) 
	\begin{pmatrix}
	u_{n-1} \\
	u_n
	\end{pmatrix},
	\begin{pmatrix}
	u_{n-1} \\
	u_n
	\end{pmatrix}
	\bigg\rangle
\]
with
\begin{equation} \label{eq:30}
	X_n(x) = 
	\prod_{j=n}^{n+N-1}
	B_j(x),
	\qquad\text{and}\qquad
	E = 
	\begin{pmatrix}
	0 & -1 \\
	1 & 0
	\end{pmatrix}.
\end{equation}
Recall that for any $X \in \GL(2, \RR)$, we have
\begin{equation}
	\label{eq:12}
	X^{-1} = -\frac{1}{\det X} E X^t E.
\end{equation}

The following proposition is the main algebraic part of the proof of the next theorem.
\begin{proposition}
	\label{prop:12}
	For each $k \geq 1$ and any generalized eigenvector $(u_n : n \in \NN_0)$ associated with $x \in \RR$ and
	$\alpha \in \RR^2 \setminus \{ 0 \}$,
	\begin{align*}
		\frac{\big| S_{n+kN} - S_n \big|}{a_{n+(k+1)N-1} (u_{n+kN-1}^2 + u_{n+kN}^2)}
		\leq
		\bigg\| \bigg( \prod_{j=1}^{k-1} X_{n+jN} \bigg)^{-1} \bigg\|
		\cdot
		\left\| \prod_{j=1}^{k} X_{n+jN}  - 
		\frac{a_{n+N-1}}{a_{n-1}} \frac{a_{n+kN-1}}{a_{n+(k+1)N-1}} \prod_{j=0}^{k-1} X_{n+jN} \right\|. 
	\end{align*}
\end{proposition}
\begin{proof}
	Using \eqref{eq:10} and \eqref{eq:30}, we can write
	\begin{align*}
		S_n 
		& = a_{n+N-1}
		\bigg\langle 
		E X_n \bigg( \prod_{j=0}^{k-1} X_{n+jN} \bigg)^{-1}
		\begin{pmatrix}
			u_{n+kN-1} \\
			u_{n+kN}
		\end{pmatrix},
		\bigg( \prod_{j=0}^{k-1} X_{n+jN} \bigg)^{-1}
		\begin{pmatrix}
			u_{n+kN-1} \\
			u_{n+kN}
		\end{pmatrix}
		\bigg\rangle \\
		& =
		a_{n+N-1}
		\bigg\langle
		\bigg(\bigg( \prod_{j=0}^{k-1} X_{n+jN} \bigg)^{-1} \bigg)^t 
		E X_n \bigg( \prod_{j=0}^{k-1} X_{n+jN} \bigg)^{-1}
		\begin{pmatrix}
			u_{n+kN-1} \\
			u_{n+kN}
		\end{pmatrix},
		\begin{pmatrix}
			u_{n+kN-1} \\
			u_{n+kN}
		\end{pmatrix}
		\bigg\rangle.
	\end{align*}
	Therefore, by \eqref{eq:12},
	\[
		S_n = a_{n+N-1} \frac{a_{n+kN-1}}{a_{n-1}}
		\bigg\langle
   		E \bigg( \prod_{j=0}^{k-1} X_{n+jN} \bigg) \bigg( \prod_{j=1}^{k-1} X_{n+jN} \bigg)^{-1}
		\begin{pmatrix}
			u_{n+kN-1} \\
			u_{n+kN}
		\end{pmatrix},
		\begin{pmatrix}
			u_{n+kN-1} \\
			u_{n+kN}
		\end{pmatrix}
		\bigg\rangle.
	\]
	Hence, 
	\[
		S_{n+kN} - S_{n} = a_{n+(k+1)N - 1} \bigg\langle E Y_n
		\bigg( \prod_{j=1}^{k-1} X_{n+jN} \bigg)^{-1}
		\begin{pmatrix}
			u_{n+kN-1} \\
			u_{n+kN}
		\end{pmatrix},
		\begin{pmatrix}
			u_{n+kN-1} \\
			u_{n+kN}
		\end{pmatrix}
		\bigg\rangle
	\]
	where
	\[
		Y_n = \bigg( \prod_{j=1}^{k} X_{n+jN} \bigg) - 
		\frac{a_{n+N-1}}{a_{n-1}} \frac{a_{n+kN-1}}{a_{n+(k+1)N-1}}
		\bigg( \prod_{j=0}^{k-1} X_{n+jN} \bigg).
	\]
	Now, by Cauchy--Schwarz inequality, we easily conclude the proof.
\end{proof}

The following theorem is the main result in this section.
\begin{theorem} 
	\label{thm:1}
	Let $N$ and $r$ be positive integers and $i \in \{ 0, 1, \ldots, N-1\}$. Let $K$ be a compact subset of
	\[
		\Lambda = \Big\{x \in \RR : 
		\lim_{k \to \infty} \discr\big(X_{kN+i}(x)\big) \text{ exists and is negative}
		\Big\}.
	\]
	If
	\begin{equation}
		\label{eq:33}
		\lim_{k \to \infty} \frac{a_{(k+1)N+i-1}}{a_{kN+i-1}} = 1,
	\end{equation}
	and
	\begin{equation}
		\label{eq:31}
		(X_{kN+i} : k \in \NN) \in \calD_{r, 0}\big(K, \GL(2, \RR)\big),
	\end{equation}
	then $(S_{kN+i} : k \in \NN)$ converges uniformly on $\sS^1 \times K$.
\end{theorem}
\begin{proof} 
	To simplify the notation, for $k \in \NN$ we set
	\begin{align*}
		S_k^{(i)} &= S_{k N + i}, &\qquad X_k^{(i)} &= X_{k N + i},
	\intertext{and}
		u^{(i)}_k &= u_{k N + i}, &\qquad a^{(i)}_k &= a_{kN + i}.
	\end{align*}
	Since $\big(\discr \big( X^{(i)}_n \big) : n \in \NN \big)$ is a sequence of polynomials with degrees $\leq 2N$, and
	$K$ is a compact subset of $\Lambda$, there is $\delta > 0$ such that for all $x \in K$ and $n \in \NN$, we have
	\footnote{We set $\sym(X) = \tfrac{1}{2}X + \tfrac{1}{2}X^t$, where $X^t$ is the transpose of $X$.}
	\[
		\det\big(\sym\big(E X_n^{(i)}\big)\big) = -\tfrac{1}{4} \discr(X_n^{(i)}) \geq \delta > 0.
	\]
	By \eqref{eq:31} the sequence $(X_n^{(i)} : n \geq 1)$ is uniformly bounded on $K$. 
	Therefore, there are $c > 0$ and $M \geq 1$ such that for all $n \geq M$ and $x \in K$,
	\begin{equation}
		\label{eq:35}
		\big| S^{(i)}_n(\alpha, x) \big| 
		\geq 
		c a_{n+1}^{(i-1)} \Big( \big( u_{n+1}^{(i-1)}(x) \big)^2 + \big( u_{n+1}^{(i)}(x) \big)^2 \Big).
	\end{equation}
	Our aim is to show that for every $\epsilon > 0$ there is $M \geq 1$ such that for all $n > m \geq M$,
	\begin{equation}
		\label{eq:36}
		\sup_{x \in K} \sup_{\alpha \in \sS^1}
		\bigg| \frac{S_{m}^{(i)}(\alpha, x)}{S_{n}^{(i)}(\alpha, x)} - 1 \bigg| < \epsilon.
	\end{equation}
	By Proposition \ref{prop:12}, for all $n > m \geq 1$,
	\[
		\big| S_{n}^{(i)}(\alpha, x) - S_{m}^{(i)}(\alpha, x) \big| 
		\leq
		a_{n+1}^{(i-1)} \cdot 
		f_{n, m}(x) \cdot
		\Big( \big( u_{n+1}^{(i-1)} \big)^2 + \big( u_{n+1}^{(i)} \big)^2 \Big)
	\]
	where
	\[
		f_{n, m}(x) = 
		\bigg\| \Big( \prod_{j=m+1}^{n-1} X^{(i)}_j(x) \Big)^{-1} \bigg\| \cdot
		\bigg\| \prod_{j=m+1}^n X_j^{(i)}(x) - 
		\frac{a^{(i-1)}_{m+1}}{a^{(i-1)}_{m}} \frac{a^{(i-1)}_{n}}{a^{(i-1)}_{n+1}}
		\prod_{j=m}^{n-1} X_{j}^{(i)}(x) \bigg\|.
	\]
	Hence, by \eqref{eq:35}, there is a constant $c>0$ such that for every $x \in K$ and $\alpha \in \sS^1$,
	\[
		\big| S_{n}^{(i)}(\alpha, x) - S_{m}^{(i)}(\alpha, x) \big| 
		\leq c f_{n, m}(x) \big|S_n^{(i)}(\alpha, x)\big|.
	\]
	We have
	\[
		f_{n, m}(x) 
		\leq 
		\bigg\| \Big( \prod_{j=m+1}^{n-1} X^{(i)}_j(x) \Big)^{-1} \bigg\| \cdot 
		\bigg(
		\Big\| \prod_{j=m+1}^n X_j^{(i)}(x) - 
		\prod_{j=m}^{n-1} X_{j}^{(i)}(x) \Big\| +
		\bigg| \frac{a^{(i-1)}_{m+1}}{a^{(i-1)}_{m}} \frac{a^{(i-1)}_{n}}{a^{(i-1)}_{n+1}} - 1 \bigg| \cdot
		\Big\| \prod_{j=m}^{n-1} X_{j}^{(i)}(x) \Big\| 
		\bigg).
	\]
	By Theorem \ref{thm:5}, the sequence $(X^{(i)}_k : k \in \NN)$ is uniformly diagonalizable. Hence, by
	Proposition \ref{prop:7}, there is $c > 0$ such that for all $x \in K$,
	\[
		f_{n, m}(x) \leq c
		\bigg( \prod_{j=m+1}^{n-1} \| D_j \|^{-1} \bigg) 
		\cdot 
		\bigg(
		\epsilon \prod_{j=m+1}^{n-1} \| D_j \| +
		\bigg| \frac{a^{(i-1)}_{m+1}}{a^{(i-1)}_{m}} \frac{a^{(i-1)}_{n}}{a^{(i-1)}_{n+1}} - 1 \bigg| 
		\cdot
		\prod_{j=m}^{n-1} \| D_j \|
		\bigg)
	\]
	for all $n > m \geq M$, provided that $M$ is large enough. Therefore,
	\[
		\sup_{x \in K} f_{n, m}(x) 
		\leq c
		\Bigg( \epsilon +
		\bigg| \frac{a^{(i-1)}_{m+1}}{a^{(i-1)}_{m}} \frac{a^{(i-1)}_{n}}{a^{(i-1)}_{n+1}} - 1 \bigg| 
		\cdot
		\sup_{x \in K} \| D_m \| 
		\Bigg)
	\]
	which, together with \eqref{eq:33}, implies \eqref{eq:36}. 

	Next, by the mean value theorem we have
	\begin{align*}
		\Big| \log \big| S_n^{(i)} \big| - \log \big| S_m^{(i)} \big| \Big| 
		&\leq 
		\Big| \big| S_n^{(i)} \big| - \big| S_m^{(i)} \big| \Big| 
		\sup_{0 \leq t \leq 1} \frac{1}{\big|S_n^{(i)}\big| + t \big(\big| S_n^{(i)} \big| - \big| S_m^{(i)}\big| \big)}\\
		&\leq
		\frac{1}{1 - \epsilon} \bigg| \frac{S_m^{(i)}}{S_n^{(i)}} - 1 \bigg|,
	\end{align*}
	thus, by \eqref{eq:36}, the sequence $\big(\log \big|S_n^{(i)}\big | : n \geq M\big)$ is a uniform Cauchy sequence of
	functions continuous on $K$. Hence, it converges to a continuous function on $K$. Since
	$\big(\big|S_n^{(i)}\big| : n \geq M \big)$ is uniformly bounded on $K$, by \eqref{eq:36},
	$\big(S_n^{(i)} : n \geq M\big)$ is a Cauchy sequence, and the theorem follows.
\end{proof}

The following corollary follows from Theorem~\ref{thm:1} and the proof of \cite[Theorem 7]{block2018}.
\begin{corollary} \label{cor:8}
Let the hypotheses of Theorem~\ref{thm:1} be satisfied. Then there is a constant $c>1$ such that for every generalized eigenvector $(u_n : n \in \NN_0)$ associated with $x \in K$, and all $n \geq 1$,
\[
	c^{-1} \big( u_0^2 + u_1^2 \big) 
	\leq 
	a_{nN+i-1} \big( u_{nN+i-1}^2 + u_{nN+i}^2 \big)
	\leq
	c \big( u_0^2 + u_1^2 \big).
\]
\end{corollary}

\section{Approximation procedure}
\label{sec:5}
In this section we describe a method that allows us to prove a formula for the density of the measure $\mu$. It is
a further development of \cite{PeriodicI}, see also \cite{GeronimoVanAssche1991}.

Let $(p_n : n \in \NN_0)$ be a sequence of polynomials satisfying the following recurrence relation
\begin{equation*} 
	\begin{gathered} 
		p_0(x) = 1, \qquad p_1(x) = \frac{x - b_0}{a_0}, \\
		a_n p_{n+1}(x) + b_n p_n(x) + a_{n-1} p_{n-1}(x) = x p_n(x), \qquad n \geq 1.
	\end{gathered}
\end{equation*}
By $\mu$ we denote a probability measure on $\RR$ such that the polynomials $(p_n : n \in \NN_0)$ are orthonormal in
$L^2(\RR, \mu)$. Let
\begin{equation}
	\label{eq:135}
	\scrD_n(x) = p_n(x) p_{n+N-1}(x) - p_{n-1}(x) p_{n+N}(x) 
\end{equation}
that is
\begin{equation}
	\label{eq:44}
	\scrD_n(x) 
	= 
	a_{n+N-1}^{-1} S_n\bigg(\bigg(1, \frac{x-b_0}{a_0} \bigg), x\bigg) 
	=
	\bigg\langle
	E X_n(x)
	\begin{pmatrix}
		p_{n-1}(x) \\
		p_{n}(x)
	\end{pmatrix},
	\begin{pmatrix}
		p_{n-1}(x) \\
		p_{n}(x)
	\end{pmatrix}
	\bigg\rangle
\end{equation}
where $X_n$ and $E$ are defined in \eqref{eq:30}. Given $L \in \NN$, we consider the truncated sequences
$(a^L_n : n \in \NN_0)$ and $(b^L_n : n \in \NN_0)$ defined by
\begin{subequations}
	\begin{equation}
	\label{eq:42a}
	a^L_n = 
	\begin{cases}
		a_n & \text{if } 0 \leq n < L+N, \\
		a_{L+i} & \text{if } L+N \leq n, \text{ and } n-L \equiv i \bmod N,
	\end{cases}
	\end{equation}
	and
	\begin{equation}
	\label{eq:42b}
	b^L_n =
	\begin{cases}
		b_n & \text{if } 0 \leq n < L+N, \\
		b_{L+i} & \text{if } L+N \leq n, \text{ and } n-L \equiv i \bmod N,
	\end{cases}
	\end{equation}
\end{subequations}
where $i \in \{0, 1, \ldots, N-1\}$. Let $(\scrD_n^L : n \in \NN_0)$ be the sequence \eqref{eq:135} associated to the
polynomials $(p_n^L : n \geq 0)$ corresponding to the sequences $a^L$ and $b^L$. Then by \eqref{eq:44}
\[
	\scrD_n^L(x) = 
	\bigg\langle
	E X_n^L(x)
	\begin{pmatrix}
		p^L_{n-1}(x) \\
		p^L_{n}(x)
	\end{pmatrix},
	\begin{pmatrix}
		p^L_{n-1}(x) \\
		p^L_{n}(x)
	\end{pmatrix}
	\bigg\rangle,
\]
where
\[
	X_n^L(x) = \prod_{j=n}^{n+N-1}
	\begin{pmatrix}
		0 & 1\\
		-\frac{a_{j-1}^L}{a_j^L} & \frac{x - b_j^L}{a_j^L}
	\end{pmatrix}.
\]
Let $\mu_L$ be the measure orthonormalizing the polynomials $(p^L_n : n \in \NN_0)$.
\begin{proposition}
	\label{prop:2}
	Let $(L_j : j \in \NN)$ be an increasing sequence of positive integers. Let $\Lambda$ be a 
	non-empty open subset of
	\[
		\bigcup_{J=1}^\infty \bigcap_{j=J}^\infty \big\{ x \in \RR : \mu'_{L_j}(x) > 0 \big\}.
	\]
	Suppose that there is a positive function $f: \Lambda \rightarrow \RR$ such that for every 
	compact subset $K \subset \Lambda$, 
	\begin{equation} 
		\label{eq:46}
		\lim_{j \to \infty} \sup_{x \in K} \big| \mu'_{L_j}(x) - f(x) \big| = 0.
	\end{equation}
	Let $\nu$ be any weak accumulation point of the sequence $(\mu_{L_j} : j \in \NN)$. Then $\nu$ 
	is a probability measure such that $(p_n : n \in \NN_0)$ are orthonormal in $L^2(\RR, \nu)$, and
	\[
		\nu({\rm d} x) = f(x) {\: \rm d} x, \qquad x \in \Lambda.
	\]
	If the moment problem for $(p_n : n \in \NN_0)$ is determinate then the measure $\nu$ is unique.
\end{proposition}
\begin{proof}
    Let us observe that, by \eqref{eq:42a} and \eqref{eq:42b}, for each $n \in \NN_0$ there is 
    $J \geq 1$ such that for all $j \geq J$, 
    \begin{equation}
        \label{eq:48}
        p^{L_j}_n(x) = p_n(x).
    \end{equation}
    In particular, each subsequence of measures $(\mu_{L_j} : j \in \NN_0)$ has $n$th 
    moment eventually constant. Let $(\mu_{L_{j_m}} : m \in \NN_0)$ be a subsequence weakly 
    converging to a measure $\nu$. By \cite[Theorem, p. 540]{Frechet1931}, $\nu$ is a 
    probability measure having all moments. Moreover, for all $k \in \NN_0$,
    \[
        \lim_{m \to \infty} \int_\RR x^k \mu_{L_{j_m}}({\rm d} x) = \int_\RR x^k \nu({\rm d} x)
    \]
    which together with \eqref{eq:48}, proves that the polynomials 
    $(p_n : n \in \NN_0)$ are orthonormal in $L^2(\RR, \nu)$.

    Let $g$ be a continuous function with a support contained in 
    $K \subset \Lambda$. Then there is $M \geq 1$ such that for all $m \geq M$, the measure 
    $\mu_{L_{j_m}}$ is absolutely continuous on $K$. Hence,
    \begin{align*}
        \bigg| \int_\RR g(x) \mu_{L_{j_m}}({\rm d} x) - \int_\RR g(x) f(x) \ud x \bigg|
        &\leq
        \int_\RR \abs{g(x)} \big| \mu_{L_{j_m}}'(x) - f(x) \big| \ud x \\
        &\leq
        \abs{K} \cdot \sup_{x \in K} |g(x)| \cdot \sup_{x \in K} \big|\mu_{L_{j_m}}'(x) - f(x)\big|
    \end{align*}
    which, by \eqref{eq:46}, implies that
    \[
        \lim_{m \to \infty} \int_\RR g(x) \mu_{L_{j_m}}({\rm d} x) = \int_\RR g(x) f(x) \ud x.
    \]
    This completes the proof.
\end{proof}

\begin{proposition}
	\label{prop:4}
	For every $L \in \NN$ and $x \in \RR$, 
	\[
		\big|\scrD_L(x) - \scrD_{L+N}^L(x)\big| 
		\leq 
		\| X_L(x) \| 
		\left|\frac{a_{L+N-1}}{a_{L-1}} - 1 \right|
		\cdot \big( p_{L+N-1}^2(x) + p_{L+N}^2(x) \big).
	\]
\end{proposition}
\begin{proof}
	By \eqref{eq:44}, we can write
	\begin{align*}
		\scrD_L(x) = 
		\bigg\langle
		E X_L(x)
		\begin{pmatrix}
			p_{L-1}(x) \\
			p_L(x)
		\end{pmatrix},
		\begin{pmatrix}
			p_{L-1}(x) \\
			p_L(x)
		\end{pmatrix}
		\bigg\rangle
		&=
		\bigg\langle
		E
		\begin{pmatrix}
			p_{L+N-1}(x) \\
			p_{L+N}(x)
		\end{pmatrix},
		X_L^{-1}(x)
		\begin{pmatrix}
			p_{L+N-1}(x) \\
			p_{L+N}(x)
		\end{pmatrix}
		\bigg\rangle \\
		&=
		\bigg\langle
		\big( X_L^{-1}(x) \big)^t
		E
		\begin{pmatrix}
			p_{L+N-1}(x) \\
			p_{L+N}(x)
		\end{pmatrix},
		\begin{pmatrix}
			p_{L+N-1}(x) \\
			p_{L+N}(x)
		\end{pmatrix}
		\bigg\rangle,
	\end{align*}
	which, by \eqref{eq:12} equals to
	\[
		\frac{a_{L+N-1}}{a_{L-1}}
		\bigg\langle
		E
		X_L(x)
		\begin{pmatrix}
			p_{L+N-1}(x) \\
			p_{L+N}(x)
		\end{pmatrix},
		\begin{pmatrix}
			p_{L+N-1}(x) \\
			p_{L+N}(x)
		\end{pmatrix}
		\bigg\rangle.
	\]
	In view of \eqref{eq:42a} and \eqref{eq:42b}, we have
	\begin{equation}
		\label{eq:7}
		\scrD_L(x) - \scrD_{L+N}^L(x) =
		\bigg\langle
		E \bigg(\frac{a_{L+N-1}}{a_{L-1}} X_L(x) - X_{L+N}^L(x) \bigg)
		\begin{pmatrix}
			p_{L+N-1}(x) \\
			p_{L+N}(x)
		\end{pmatrix},
		\begin{pmatrix}
			p_{L+N-1}(x) \\
			p_{L+N}(x)
		\end{pmatrix}
		\bigg\rangle.	
	\end{equation}
	Moreover, since
	\[
		\Id - B_L^{-1}(x) 
		\begin{pmatrix}
			0 & 1 \\
			-\frac{a_{L+N-1}}{a_L} & \frac{x - b_L}{a_L}
		\end{pmatrix}
		=
		\begin{pmatrix}
			1 - \frac{a_{L+N-1}}{a_{L-1}} & 0 \\
			0 & 0
		\end{pmatrix},
	\]
	we obtain
	\begin{align*}
		X_L(x) - X^L_{L+N}(x) 
		&= X_L(x)
		\bigg(\Id - B_L^{-1}(x) 
		\begin{pmatrix}
			0 & 1 \\
			-\frac{a_{L+N-1}}{a_L} & \frac{x - b_L}{a_L}
		\end{pmatrix}	
		\bigg) \\
		&= X_L(x) 
		\begin{pmatrix}
			1 - \frac{a_{L+N-1}}{a_{L-1}} & 0 \\
			0 & 0
		\end{pmatrix}.
	\end{align*}
	Hence,
	\[
		\frac{a_{L+N-1}}{a_{L-1}} X_L(x) - X_{L+N}^L(x) 
		=
		X_L(x)
		\begin{pmatrix}
			0 & 0 \\
			0 & \frac{a_{L+N-1}}{a_{L-1}} - 1
		\end{pmatrix},
	\]
	which together with \eqref{eq:7} concludes the proof.
\end{proof}

\begin{corollary}
	\label{cor:3}
	For all $x \in \RR$ and $L \in \NN$, 
	\[
		\big\|
		X_L(x) - X_{L+N}^L(x)
		\big\|
		\leq
		\|X_L(x) \| \cdot \bigg|\frac{a_{L+N-1}}{a_{L-1}} - 1\bigg|.
	\]
\end{corollary}

The next theorem is the main result in this section.
\begin{theorem}
	\label{thm:6}
	Let $N$ and $r$ be positive integers. Let $(L_j : j \in \NN)$ be an increasing sequence of
	positive integers. Let $K$ be a compact subset of 
	\[
		\Lambda = \Big\{x \in \RR : \lim_{j \to \infty} \discr \big( X_{L_j}(x) \big) 
		\text{ exists and is negative} \Big\}.
	\]
	Assume that
	\begin{equation}
		\label{eq:50}
		\lim_{j \to \infty} \frac{a_{L_j+N-1}}{a_{L_j-1}} = 1.
	\end{equation}
	Suppose that there is a positive function $g: K \rightarrow \RR$ such that 
	\begin{equation}
		\label{eq:51}
		\lim_{j \to \infty} \sup_{x \in K} \Big| a_{L_j+N-1} \big| \scrD_{L_j}(x) \big| - g(x) \Big| = 0,
	\end{equation}
	and 
	\begin{equation}
		\label{eq:52}
		\sup_{j \in \NN} \sup_{x \in K} \| X_{L_j}(x) \| < \infty.
	\end{equation}
	Let $\nu$ be any weak accumulation point of the sequence $(\mu_{L_j} : j \in \NN)$. 
	Then $\nu$ is a probability measure such that $(p_n : n \in \NN_0)$ are orthogonal in 
	$L^2(\RR, \nu)$, which is absolutely continuous on $K$ with the density
	\[
		\nu'(x) = \frac{\sqrt{-h(x)}}{2 \pi g(x)}, \qquad x \in K
	\]
	where
	\begin{equation}
		\label{eq:17}
		h(x) = \lim_{j \to \infty} \discr \big( X_{L_j}(x) \big), \qquad x \in K.
	\end{equation}
\end{theorem}
\begin{proof}
	For a positive integer $L$ such that $L \in \{ L_j : j \in \NN \}$, we set
	\[
		\Lambda_L = \Big\{ x \in \RR: \discr\big(X^L_{L+N}(x)\big) < 0 \Big\},
	\]
	and
	\[
		S^L_n(x) = a^L_{n+N-1} \scrD^L_{n}(x), \qquad n \geq 1.
	\]
	In view of \cite[Theorem 3]{PeriodicII}, (see also \cite[Theorem 6]{GeronimoVanAssche1991}), 
	for each $x \in \Lambda_L$,
	\[
		\lim_{k \to \infty} \big|S^L_{L+kN}(x)\big| \quad \text{exists}
	\]
	and defines a positive continuous function $g^L: \Lambda_L \rightarrow \RR$. Moreover, 
	the measure $\mu_L$ is absolutely continuous on $\Lambda_L$ with the density
	\begin{equation}
		\label{eq:55}
		\mu'_L(x) = \frac{\sqrt{-\discr\big(X^L_{L+N}(x)\big)}}{2 \pi g^L(x)}.
	\end{equation}
	By Proposition \ref{prop:12}, we have
	\begin{align*}
		&\big| S^L_{n+N}(x) - S^L_{n}(x) \big| \\
		&\qquad
		\leq 
		a_{n+N-1}^L
		\Big( \big( p^L_{n+N-1}(x) \big)^2 + \big( p^L_{n+N-1}(x) \big)^2 \Big)
		\bigg\|
			\frac{a^L_{n+2N-1}}{a^L_{n+N-1}} X^L_{n+N}(x) -
			\frac{a^L_{n+N-1}}{a^L_{n-1}} X^L_{n}(x)
		\bigg\|.
	\end{align*}
	Hence, by \eqref{eq:42a} and \eqref{eq:42b}, we conclude that $S^L_{n+N}(x) = S^L_n(x)$ for all 
	$n \geq L+1$. Thus, for all $x \in \Lambda_L$,
	\begin{equation}
		\label{eq:11}
		g^L(x) = \big| S^L_{L+N}(x) \big|.
	\end{equation}
	Next, let us observe that there is $c > 0$ such that for all $A, B \in \Mat(2, \RR)$,
	\[
		| \discr A - \discr B | \leq c (\|A\|+\|B\|)\|A - B\|.
	\]
	Therefore, by Corollary \ref{cor:3}, we obtain 
	\begin{equation}
		\label{eq:34}
		\Big| \discr\big(X_{L+N}^L(x)\big) - \discr\big(X_L(x)\big) \Big|
		\leq
		c \left\|X_L(x) \right\|^2 \bigg|\frac{a_{L+N-1}}{a_{L-1}} - 1 \bigg|.
	\end{equation}
	Let us fix a compact subset $K \subset \Lambda$. Since $\discr \big( X_L(x) \big)$ is 
	a polynomial of degree at most $2N$, the convergence in \eqref{eq:17} is uniform on $K$. 
	Thus, by \eqref{eq:50}, \eqref{eq:52}, and \eqref{eq:34} we get
	\begin{equation}
		\label{eq:39}
		\lim_{j \to \infty}
		\sup_{x \in K}
		\Big| \discr \Big( X_{L_j+N}^{L_j}(x) \Big) - h(x) \Big| = 0.
	\end{equation}
	In particular, $K \subset \Lambda_L$ for all $L$ sufficiently large. Now, setting
	\[
		S_n(x) = a_{n+N-1} \scrD_n(x),
	\]
	by Proposition \ref{prop:4}, we get
	\begin{align*}
		\big| S^L_{L+N}(x) - S_L(x) \big| 
		&= 
		a_{L+N-1} \big|\scrD^L_{L+N}(x) - \scrD_L(x)\big| \\
		&\leq
		a_{L+N-1} \big( p_{L+N-1}^2(x) + p_{L+N}^2(x) \big)
		\| X_L(x) \| \left|\frac{a_{L+N-1}}{a_{L-1}} - 1 \right|.
	\end{align*}
	Since $K$ is a compact subset of $\Lambda$, there is $L'$ such that for all $L \geq L'$ and 
	$x \in K$ we have
	\[
		\det \Big( \sym \big( E X_L(x) \big) \Big) = - \frac{1}{4} \discr \big( X_L(x) \big) > 0,
	\]
	which together with \eqref{eq:52} implies that there are $c > 0$ and $L' \geq 1$ such that for 
	all $L \geq L'$ and $x \in K$,
	\[
		|S_L(x)| \geq c^{-1} a_{L+N-1} \big( p_{L+N-1}^2(x) + p_{L+N}^2(x) \big).
	\]
	Hence,
	\[
		\big|S^L_{L+N}(x) - S_L(x)\big| 
		\leq c 
		|S_L(x)| \cdot \| X_L(x) \| \left|\frac{a_{L+N-1}}{a_{L-1}} - 1 \right|,
	\]
	which, by \eqref{eq:50}, \eqref{eq:51} and \eqref{eq:11}, gives
	\begin{equation}
		\label{eq:57}
		\lim_{j \to \infty}
		\sup_{x \in K} \big| g^{L_j}(x) - g(x) \big| = 0.
	\end{equation}
	Finally, by \eqref{eq:55}, \eqref{eq:39}, and \eqref{eq:57} we obtain
	\[
		\lim_{j \to \infty}
		\sup_{x \in K} \bigg| \mu'_{L_j}(x) - \frac{\sqrt{-h(x)}}{2\pi g(x)} \bigg| = 0,
	\]
	and the theorem follows by Proposition \ref{prop:2}.
\end{proof}

\begin{corollary} \label{cor:9}
Let the hypotheses of Theorem~\ref{thm:1} be satisfied. Then there is a positive function
$g: K \rightarrow \RR$, such that
\begin{equation} \label{eq:154}
	\lim_{\stackrel{n \to \infty}{n \equiv i \bmod N}} 
	\sup_{x \in K}
	\Big|
	a_{n+N-1} \big|\scrD_n (x) \big| - g(x)
	\Big|
	=0.
\end{equation}
Moreover, there is a probability measure $\nu$ such that $(p_n : n \in \NN_0)$ are orthonormal 
in $L^2(\RR, \nu)$, which is absolutely continuous on $K$ with the density
\begin{equation} \label{eq:155}
	\nu'(x) = \frac{\sqrt{-h(x)}}{2\pi g(x)}, \qquad x \in K,
\end{equation}
where
\[
	h(x) = \lim_{\stackrel{n \to \infty}{n \equiv i \bmod N}} \discr \big( X_{n}(x) \big), 
	\qquad x \in K.
\]
Furthermore, 
\[
	\lim_{k \to \infty} \sup_{x \in K} |\nu'(x) - \mu_{kN+i}'(x)| = 0,
\]
where $\mu_L$ is the sequence orthonormalizing the sequence $(p_n^L : n \in \NN_0)$.
If the moment problem for $(p_n : n \in \NN_0)$ is determinate, then the measure $\nu$
is unique.
\end{corollary}
\begin{proof}
For $j \geq 1$ set $L_j = j N + i$. In view of \eqref{eq:135} and \eqref{eq:44}, Theorem~\ref{thm:1} gives
\eqref{eq:154}.

Since $(X_{nN+i} : n \in \NN)$ belongs to $\calD_{r, 0}(K, GL(2, \RR))$ it is uniformly bounded. Hence,
Theorem~\ref{thm:6} gives \eqref{eq:155}. The assertion of uniqueness of $\nu$ follows from
Proposition~\ref{prop:2}. The proof is complete.
\end{proof}

\section{The exact asymptotic of orthogonal polynomials} 
\label{sec:6}
Let $N$ be a positive integer and $i \in \{0, 1, \ldots, N-1\}$. In this section we prove an asymptotic formula for
orthonormal polynomials $(p_{kN+i} : k \in \NN_0)$ under the condition that $(X_{k N + i} : k \in \NN_0)$ belongs to
the Stolz class $\calD_{r, 0}\big(K, \GL(2, \RR)\big)$ for some compact set $K$. The following 
theorem is a generalization of \cite[Theorem 4]{GeronimoVanAssche1991}.
\begin{theorem} 
	\label{thm:2}
	Let $N$ and $r$ be positive integers and $i \in \{0, 1, \ldots, N-1\}$. Let $K$ be a compact subset of
	\[
		\Lambda = \Big\{x \in \RR : 
		\lim_{k \to \infty} \discr\big(X_{kN+i}(x)\big) \text{ exists and is negative}
		\Big\}.
	\]
	If
	\[
		\lim_{k \to \infty} \frac{a_{(k+1)N+i-1}}{a_{kN+i-1}} = 1,
	\]
	and
	\[
		(X_{kN+i} : k \in \NN) \in \calD_{r, 0}\big(K, \GL(2, \RR) \big),
	\]
	then there are continuous functions $t_j : K \rightarrow \CC$ and $\vphi: K \rightarrow \CC$ such that
	\[
		\inf_{j \in \NN} \inf_{x \in K} \abs{t_j(x)} > 0
	\]
	and
	\[
		\lim_{k \to \infty} 
		\sup_{x \in K}
		\bigg| 
		\Big(p_{(k+1) N+i}(x) -\overline{\lambda_{kN+i}(x)} p_{kN+i}(x) \Big) 
		\prod_{j=1}^k \frac{1}{t_j(x)} - \vphi(x) \bigg| = 0.
	\]
	Moreover,
	\begin{equation} 
		\label{eq:143}
		\lim_{j \to \infty} \sup_{K} \big|t_j - \lambda_{jN+i} \big| = 0
	\end{equation}
	where $\lambda_n$ is defined in \eqref{eq:153}.
\end{theorem}
\begin{proof}
	By Theorem \ref{thm:5}, the sequence $(X_{kN+i} : k \in \NN_0)$ is uniformly diagonalizable on $K$.
	Let $(D_k : k \geq M)$ be the corresponding sequence of diagonal matrices where
	\[
		D_k = 
		\begin{pmatrix}
			\gamma_k & 0 \\
			0 & \overline{\gamma_k}
		\end{pmatrix}.
	\]
	We define
	\[
		t_k = 
		\begin{cases}
			1 & \text{if } k \leq M, \\
			\gamma_{k} & \text{otherwise.}
		\end{cases}
	\]
	Then \eqref{eq:143} easily follows from \eqref{eq:152}. Let us notice that there is $\delta > 0$, such that for all
	$j \geq 1$ and $x \in K$, $\abs{t_j(x)} \geq \delta$. For $x \in K$, we set
	\[
		\phi_k(x) = \frac{p_{(k+1)N+i}(x) - \overline{\lambda_{kN+i}} p_{kN+i}}
		{\prod_{j = M+1}^{k} t_j(x)}.
	\]
	Given $\epsilon > 0$, we select $L > M$ such that
	\begin{equation}
		\label{eq:45}
		\sum_{k = L-1}^\infty \sup_{x \in K} \big\|\Delta C_k(x) \big\| < \epsilon,
	\end{equation}
	and
	\begin{equation}
		\label{eq:47}
		\sup_{k \geq L-1} \sup_{x \in K} \big\| Q_\infty(x) - Q_k(x) \big\| < \epsilon.
	\end{equation}
	First, we replace the polynomials $(p_{kN+i} : k \geq L)$ by the sequence of functions $(q_k : k \geq L)$, where
	\begin{equation}
		\label{eq:49}
		q_k(x) = 
		\bigg\langle
    	Q_{\infty} \Big( \prod_{j=L}^{k-1} D_j \Big) Q_{L-1}^{-1}
	    \begin{pmatrix}
			p_{LN+i-1} \\
			p_{LN+i}
		\end{pmatrix},
		\begin{pmatrix}
		    0 \\
		    1
		\end{pmatrix}
		\bigg\rangle.
	\end{equation}
	For $x \in K$ we set
	\[
		\psi_k(x) = \frac{q_{k+1}(x) - \overline{\lambda_{kN+i}(x)} q_{k}(x)}
		{\prod_{j=M+1}^{k} t_j(x)}.
	\]
	We claim the following holds true.
	\begin{claim}
		\label{clm:1}
		There is $c > 0$ such that for all $k \geq L$, 
		\[
			\sup_{x \in K} \frac{\big| p_{kN+i}(x) - q_k(x) \big|}{\prod_{j=M+1}^{k-1} \abs{t_j(x)}} 
			\leq c \epsilon.
		\]
	\end{claim}
	For the proof, let us observe that the recurrence relation implies that for every $j \geq 1$,
	\[
		\begin{pmatrix}
			p_{(j+1)N+i-1} \\
			p_{(j+1)N+i}
		\end{pmatrix}
		=
		X_{jN+i}
		\begin{pmatrix}
			p_{jN+i-1} \\
			p_{jN+i}			
		\end{pmatrix},
	\]
	thus, for $k > L$,
	\[
		p_{kN+i} =
		\bigg\langle
			Q_{k-1} \Big( \prod_{j=L}^{k-1} D_j C_j^{-1} C_{j-1} \Big) Q_{L-1}^{-1}
			\begin{pmatrix}
				p_{LN+i-1} \\
				p_{LN+i}
			\end{pmatrix},
			\begin{pmatrix}
				0 \\
				1
			\end{pmatrix}
		\bigg\rangle.
	\]
	Therefore,
	\begin{equation} 
		\label{eq:76}
		p_{kN+i} - q_{k} =
		\bigg\langle
		Y_k
		\begin{pmatrix}
			p_{LN+i-1} \\
			p_{LN+i}
		\end{pmatrix},
		\begin{pmatrix}
			0 \\
			1
		\end{pmatrix}
		\bigg\rangle
	\end{equation}
	where
	\[
		Y_k = Q_{k-1} \Big(\prod_{j=L}^{k-1} D_j C_j C_{j-1}^{-1} - \prod_{j=L}^{k-1} D_j \Big) Q_{L-1}^{-1}
		+ (Q_{k-1} - Q_{\infty}) \Big( \prod_{j=L}^{k-1} D_j \Big) Q_{L-1}^{-1}.
	\]
	By \eqref{eq:22}, we have
	\[
		\| Y_k \| 
		\leq c
		\bigg( \prod_{j=L}^{k-1} \| D_j \| \bigg)
		\bigg( 
		\sum_{j=L-1}^\infty \sup_K \| \Delta C_j \| + \| Q_{k-1} - Q_{\infty} \| \bigg).
	\]
	Since $\|D_j\| = \abs{t_j}$ for $j \geq M$, by \eqref{eq:45} and \eqref{eq:47}, we conclude that
	\[
		\| Y_{k} \| \leq c' \epsilon \prod_{j = L}^{k-1} \abs{t_j}.
	\]
	Hence, by \eqref{eq:76} we get
	\[
		\frac{| p_{kN+i} - q_k |}{\prod_{j=M+1}^{k-1} \abs{t_j}}
		\leq c \epsilon
		\frac{\sqrt{p_{LN+i-1}^2 + p_{LN+i}^2}}{\prod_{j=M+1}^{L-1} \abs{t_j}}.
	\]
	The task now is to show
	\begin{claim}
		\label{clm:2}
		\begin{equation}
			\label{eq:32}
			\sup_{x \in K} \frac{\sqrt{p_{LN+i-1}^2(x) + p_{LN+i}^2(x)}}{\prod_{j=M+1}^{L-1} \abs{t_j(x)}} \leq c.
		\end{equation}
	\end{claim}
	By the recurrence relation we have
	\begin{align*}
		\begin{pmatrix}
			p_{LN+i-1} \\
			p_{LN+i}
		\end{pmatrix}
		&=
		\bigg( \prod_{j=M+1}^{L-1} X_{jN+i} \bigg)
		\begin{pmatrix}
			p_{(M+1)N+i-1} \\
			p_{(M+1)N+i}
		\end{pmatrix} \\
		&=
		Q_{L-1} \bigg( \prod_{j=M+1}^{L-1} D_{j} C_{j}^{-1} C_{j-1} \bigg) Q_{M}^{-1}
		\begin{pmatrix}
			p_{(M+1)N+i-1} \\
			p_{(M+1)N+i}
		\end{pmatrix}.
	\end{align*}
	Hence, by \eqref{eq:21},
	\[
		\sqrt{p_{LN+i-1}^2 + p_{LN+i}^2} 
		\leq c 
		\bigg( \prod_{j=M+1}^{L-1} |t_j| \bigg) \sqrt{p_{(M+1)N+i-1}^2 + p_{(M+1)N+i}^2},
	\]
	and so,
	\begin{equation}
		\label{eq:75}
		\frac{\sqrt{p_{LN+i-1}^2 + p_{LN+i}^2}}{\prod_{j=M+1}^{L-1} |t_j|} 
		\leq c' 
		\sqrt{p_{(M+1)N+i-1}^2 + p_{(M+1)N+i}^2}.
	\end{equation}
	Since the right-hand side of \eqref{eq:75} is a continuous function on $K$, it is uniformly bounded. 
	This proves \eqref{eq:32} and Claim~\ref{clm:1} follows.

	As a consequence of Claim \ref{clm:1}, for $k > L > M$ we easily get
	\begin{align*}
		\big|
		\phi_k
		-
		\psi_k
		\big|
		&\leq
		c \epsilon \big(1 + \delta^{-1} \abs{\lambda_{kN+i}}\big).
	\end{align*}
	Since $(\lambda_{kN+i} : k \in \NN_0)$ converges to a continuous function on $K$, for any $n \geq m > L > M$ we obtain
	\begin{align*}
		\big|
		\phi_n - \phi_m
		\big|
		\leq
		c\epsilon
		+
		\big|
		\psi_n - \psi_m
		\big|.
	\end{align*}
	Therefore, our task is reduced to showing that the sequence $(\psi_k : k \geq L)$ converges
	uniformly on $K$. To do so, using \eqref{eq:49} we write 
	\[
		q_{k+1} - \overline{\lambda_{kN+i}} q_{k} =
		\bigg\langle
    	Q_{\infty} 
		\big(D_{k} - \overline{\lambda_{kN+i}} \Id \big) \Big( \prod_{j=L}^{k-1} D_j \Big) 
		Q_{L-1}^{-1}
	    \begin{pmatrix}
			p_{LN+i-1} \\
			p_{LN+i}
		\end{pmatrix},
		\begin{pmatrix}
		    0 \\
		    1
		\end{pmatrix}
		\bigg\rangle.
	\]
	Observe that
	\[
		\frac{1}{\prod_{j=L}^{k} t_j} 
		\big(D_k - \overline{\lambda_{k N + i}} \Id \big) \Big( \prod_{j=L}^{k-1} D_j \Big)
		=
		\begin{pmatrix}
			\frac{\gamma_k  - \overline{\lambda_{kN+i}}}{\gamma_k}  & 0 \\
			0 & \frac{\overline{\gamma_k} - \overline{\lambda_{kN+i}}}{\gamma_k} 
			\prod_{j = L}^{k-1} \frac{\overline{\gamma_k}}{\gamma_k}
		\end{pmatrix}.
	\]
	Since the sequence $(\lambda_{kN+i} : k \in \NN_0)$ converges uniformly on $K$, and
	\[
		\lim_{k \to \infty} \sup_{x \in K} \big|\gamma_k(x) - \lambda_{kN+i}(x) \big| = 0,
	\]
	we arrive at the conclusion that the sequence $(\psi_k : k \geq L)$ converges uniformly on $K$.
	This completes the proof of the theorem.
\end{proof}

Our aim is to deduce from Theorem \ref{thm:2} the asymptotic behavior of the polynomials $(p_{kN+i} : k \in \NN_0)$.
\begin{proposition}
	\label{prop:8}
	Fix $i \in \{ 0, 1, \ldots, N-1 \}$ and let $K$ be a compact subset of
	\[
		\Lambda = \Big\{x \in \RR : 
		\lim_{n \to \infty} \discr\big(X_{nN+i}(x)\big) \text{ exists and is negative}
		\Big\}.
	\]
	Suppose that there are continuous functions $t_j : K \rightarrow \CC$, and $\vphi: K \rightarrow \CC$, such that
	\[
		\inf_{j \in \NN} \inf_{x \in K} \abs{t_j(x)} > 0,
	\]
	and
	\[
		\lim_{n \to \infty} \sup_{x \in K} 
		\bigg|
		\frac{p_{(n+1)N+i}(x) - \overline{\lambda_{nN+i}(x)} p_{nN+i}(x)}{\prod_{j=1}^n t_j(x)} - \vphi(x)
		\bigg| = 0.
	\]
	Then 
	\[
		\lim_{n \to \infty} \sup_{x \in K} 
		\bigg|
		\frac{1}{2} \sqrt{-\discr\big(X_{nN+i}(x)\big)} \frac{p_{nN+i}(x)}{\prod_{j=1}^{n} \abs{t_j(x)}}
		- 
		|\varphi(x)| \sin \Big( \sum_{j=1}^{n} \arg t_j(x) + \arg \varphi(x) \Big) 
		\bigg| = 0.
	\]
\end{proposition}
\begin{proof}
	We have
	\[
		\lim_{n \to \infty}
		\sup_{x \in K}
		\bigg|
		\frac{p_{(n+1)N+i}(x) - \overline{\lambda_{nN+i}(x)} p_{nN+i}(x)}{\prod_{j=1}^n \abs{t_j(x)}} 
		- \vphi(x) \prod_{j=1}^n \frac{t_j(x)}{\abs{t_j(x)}}\bigg| = 0.
	\]
	Since polynomials $p_n$ are having real coefficients, by taking imaginary part we arrive at
	\[
		\lim_{n \to \infty}
		\sup_{x \in K}
		\bigg|
		\Im \big( \lambda_{nN+i}(x) \big) \frac{p_{nN+i}(x)}{\prod_{j=1}^n {\abs{t_j(x)}}}
		- \abs{\vphi(x)}\sin\Big(\sum_{j=1}^n \arg t_j(x) + \arg \vphi(x) \Big)
		\bigg|=0.
	\]
	Finally,
	\[
		\Im(\lambda_{nN+i}) = \frac{1}{2} \sqrt{-\discr\big(X_{nN+i}\big)},
	\]
	and the conclusion follows.
\end{proof}

Our next task is to compute $\abs{\vphi(x)}$. To do this, once again, we use the truncated sequences
defined in \eqref{eq:42a} and \eqref{eq:42b}. 
\begin{theorem}
	\label{thm:3}
	Let $N$ and $r$ be positive integers and $i \in \{0, 1, \ldots, N-1\}$. Let $K$ be a compact subset of 
	\[
		\Lambda = \Big\{x \in \RR : 
		\lim_{n \to \infty} \discr\big(X_{nN+i}(x)\big) \text{ exists and is negative}
		\Big\}.
	\]
	Assume that
	\begin{equation} \label{eq:139}
		\lim_{n \to \infty} \frac{a_{(n+1)N+i-1}}{a_{nN+i-1}} = 1,
	\end{equation}
	and
	\[
		(X_{nN+i} : n \in \NN) \in \calD_{r, 0} \big( K, \GL(2, \RR) \big).
	\]
	Suppose that $\calX: K \rightarrow \GL(2, \RR)$ is the limit of $(X_{nN+i} : n \in \NN)$. Then
	\begin{align} \label{eq:140}
		&
		\lim_{n \to \infty}
		\sup_{x \in K}
		\bigg|
		\sqrt{a_{(n+1)N+i-1}} \sqrt{-\discr X_{nN+i}(x)} p_{nN+i}(x) \\
		&\qquad\qquad\qquad-
		\sqrt[4]{-\discr \calX(x)} \sqrt{\frac{2 \abs{\calX_{21}(x)}}{\pi \nu'(x)}}
		\sin \Big(
		\sum_{j = 1}^{n} \arg t_j(x) + \vphi(x)
		\Big)
		\bigg| = 0 \nonumber
	\end{align}
	where $\nu$ is the measure defined in Theorem \ref{thm:6}.
\end{theorem}
\begin{proof}
	Since $K \subset \Lambda$, and $\discr X_{kN+i}$ is a polynomial of degree at most $2N$, there 
	are $\delta > 0$ and $M \geq 1$ such that for all $x \in K$ and $n \geq M$,
	\[
		\discr X_{nN+i}(x) \leq -\delta.
	\]
	Given $L = kN+i$, we set
	\[
		\Lambda_L = \Big\{x \in \RR : \discr\big(X^L_{L+N}(x) \big) < 0 \Big\}.
	\]
	In view of \eqref{eq:42a} and \eqref{eq:42b}, we have
	\[
		X^L_{jN+i} = 
		\begin{cases}
			X_{jN+i} & \text{if } 0 \leq j \leq k, \\
			X_{L+N}^L & \text{if } k < j. 
		\end{cases}
	\]
	Moreover, there is $L_0 \geq M$ such that $K \subset \Lambda_L$ for all $L \geq L_0$. 
	Thus, for all $L > L_0$, $j \geq M$, and $x \in K$,
	\[
		\discr \big( X^L_{jN+i}(x) \big) < 0. 
	\]
	In particular, by Theorem \ref{thm:5}, the sequence $\big( X_{mN+i}^L : m \geq M \big)$ 
	is uniformly diagonalizable with
	\[
		D^L_m = 
		\begin{pmatrix}
			\gamma_m^L & 0 \\
			0 & \overline{\gamma_m^L}
		\end{pmatrix}.
	\]
	We set
	\[
		t_m^L(x) = 
		\begin{cases}
			1 & \text{if } 0 \leq m \leq M, \\
			\gamma_m^L(x) & \text{if } M < m,
		\end{cases}
	\]
	and
	\begin{align*}
		\phi_{mN+i}(x) &= \frac{p_{(m+1)N+i}(x) - \overline{\lambda_{mN+i}(x)} p_{mN+i}(x)}
		{\prod_{j=M+1}^m t_j(x)}, \\
		\phi_{mN+i}^L(x) &= \frac{p^L_{(m+1)N+i}(x) - \overline{\lambda_{mN+i}^L(x)} p^L_{mN+i}(x)}
		{\prod_{j = M+1}^m t_j^L(x)}.
	\end{align*}
	We next show the following claim.
	\begin{claim}
		\label{clm:5}
		For all $m \geq k+1$, we have $\phi_{mN+i}^L = \phi_{L+N}^L$.
	\end{claim}
	For the proof, let us first observe that
	\[
		X^L_{mN+i} = X^L_{L+N} \qquad\text{for all } m \geq k+1.
	\]
	Hence, by Remark \ref{rem:1}, for $m \geq k+1$, 
	\begin{align*}
		p_{mN+i}^L 
		&= 
		\left\langle
		\big(X^L_{L+N}\big)^{m-k-1} 
		\begin{pmatrix}
			p^L_{L+N-1} \\
			p^L_{L+N}
		\end{pmatrix},
		\begin{pmatrix}
			0 \\
			1
		\end{pmatrix}
		\right\rangle \\
		&=
		\left\langle
		C^L \big(D^L\big)^{m-k-1} \big(C^L\big)^{-1}
		\begin{pmatrix}
			p^L_{L+N-1} \\
			p^L_{L+N}
		\end{pmatrix},
		\begin{pmatrix}
			0 \\
			1
		\end{pmatrix}
		\right\rangle 
	\end{align*}
	where
	\[
		X^L_{L+N} = 
		C^L D^L \big(C^L\big)^{-1},
		\qquad\text{and}\qquad
		D^L = 
		\begin{pmatrix}
			\lambda^L & \\
			0 & \overline{\lambda^L}
		\end{pmatrix}.
	\]
	Therefore, for $m \geq k+1$,
	\begin{align*}
		\phi_{mN+i}^L 
		&= 
		\frac{p_{(m+1)N+i}^L - \overline{\lambda^L} p_{mN+i}^L}
		{\prod_{j = M+1}^m t_j^L} \\
		&=
		\frac{1}{\prod_{j = M+1}^m t_j^L}
		\left\langle
		C^L
		\big(D^L - \overline{\lambda^L} \Id\big) \big(D^L\big)^{m-k-1}
		\big(C^L\big)^{-1}
		\begin{pmatrix}
			p^L_{L+N-1} \\
			p^L_{L+N}
		\end{pmatrix},
		\begin{pmatrix}
			0 \\
			1
		\end{pmatrix}
		\right\rangle.
	\end{align*}
	Since
	\[
		\frac{1}{\big( \lambda^L \big)^{m-k-1}} 
		\big(D^L - \overline{\lambda^L} \Id\big) \big(D^L\big)^{m-k-1} =
		\begin{pmatrix}
			\lambda^L - \overline{\lambda^L} & 0 \\
			0 & 0
		\end{pmatrix}
		\begin{pmatrix}
			1 & 0 \\
			0 & \Big( \frac{\overline{\lambda^L}}{\lambda^L} \Big)^{m-k-1}
		\end{pmatrix}
		=
		D^L - \overline{\lambda^L} \Id,
	\]
	we obtain
	\begin{align*}
		\phi^L_{mN+i}
		=
		\frac{1}{\prod_{j = M+1}^{k+1} t_j^L}
		\left\langle
		C^L\big(D^L - \overline{\lambda^L} \Id\big)
		\big(C^L\big)^{-1}
		\begin{pmatrix}
			p^L_{L+N-1} \\
			p^L_{L+N}
		\end{pmatrix},
		\begin{pmatrix}
			0 \\
			1
		\end{pmatrix}
		\right\rangle
		=
		\phi_{L+N}^L,
	\end{align*}
	and the claim follows.
	\begin{claim}
		\label{clm:4}
		Let  $L_k = k N + i$. If
		\begin{equation}
			\label{eq:56}
			\lim_{k \to \infty} \sup_{x \in K} \Big\| X^{L_k}_{L_k+N} - X_{L_k + N} \Big\| = 0,
		\end{equation}
		then
		\[
			\lim_{k \to \infty} \sup_{x \in K}
			\big|\phi^{L_k}_{L_k+N}(x) - \phi_{L_k+N}(x) \big| = 0.
		\]
	\end{claim}
	Since $p_n^{L_k} = p_n$ for $n \leq L_k+N$, we have
	\[
		\phi_{L_k+N} = \frac{1}{\prod_{j = M+1}^{k+1} t_j}
		\left\langle
		\Big(X_{L_k+N} - \overline{\lambda_{L_k+N}} \Id\Big)
		\begin{pmatrix}
			p_{L_k+N-1} \\
			p_{L_k+N}
		\end{pmatrix},
		\begin{pmatrix}
			0 \\
			1
		\end{pmatrix}
		\right\rangle
	\]
	and
	\[
		\phi_{L_k+N}^{L_k} = \frac{1}{\prod_{j = M+1}^{k+1} t_j^{L_k}}
		\left\langle
		\Big(X_{L_k+N}^{L_k} - \overline{\lambda^{L_k}} \Id \Big)
		\begin{pmatrix}
			p_{L_k+N-1} \\
			p_{L_k+N}
		\end{pmatrix},
		\begin{pmatrix}
	        0 \\
	        1 
		\end{pmatrix}
		\right\rangle.
	\]
	Observe that for $j \in \NN$, we have
	\begin{equation}
		\label{eq:65}
		t_j^{L_k} = 
		\begin{cases}
			t_j & \text{if } j \leq k,\\
			\lambda^{L_k} & \text{if } j > k.
		\end{cases}
	\end{equation}
	Hence,
	\[
		\phi_{L_k+N} - \phi^{L_k}_{L_k+N} = \frac{1}{\prod_{j=M+1}^k t_j} 
		\left\langle
		Y_k 
		\begin{pmatrix}
			p_{L_k+N-1} \\
			p_{L_k+N}
		\end{pmatrix},
		\begin{pmatrix}
			0 \\
			1
		\end{pmatrix}
		\right\rangle 
	\]
	where
	\[
		Y_k = 
		\frac{1}{\gamma_{k+1}} 
		\Big(X_{L_k+N} - \overline{\lambda_{L_k+N}} \Id \Big) 
		- \frac{1}{\lambda^{L_k}}
		\Big(X_{L_k+N}^{L_k} - \overline{\lambda^{L_k}} \Id\Big).
	\]
	Since
	\begin{align*}
		\| Y_k \| 
		\leq 
		\bigg|\frac{1}{\gamma_{k+1}} - \frac{1}{\lambda^{L_k}} \bigg|
		\big\|X_{L_k+N} - \overline{\lambda_{L_k+N}} \Id\big\|
		+
		\frac{1}{|\lambda^{L_k}|}
		\big\|X_{L_k+N}^{L_k} - X_{L_k+N} \big\| 
		+ 
		\frac{|\lambda^{L_k} - \lambda_{L_k+N}|}{|\lambda^{L_k}|},
	\end{align*}
	by \eqref{eq:56} and Claim \ref{clm:2}, we conclude that
	\[
		\lim_{k \to \infty}
		\sup_{x \in K} 
		\big|
		\phi_{L_k+N}(x) - \phi^{L_k}_{L_k+N}(x)
		\big|
		\leq
		c \lim_{k \to \infty} \sup_{x \in K} \big\| Y_k(x) \big\| = 0,
	\]
	proving the claim.

	Our next goal is to compute $\abs{\phi^L_{L+N}(x)}^2$.
	\begin{claim}
		\label{clm:3}
		For $L = kN + i$ and $x \in K$, we have
		\[
			\big| \phi^L_{L+N}(x) \big|^2
			=
			\frac{\left|\big[X_{L+N}^L(x)\big]_{2,1}\right|}
			{a_{(M+1)N+i-1} \det Q_k^{-1}(x) \det Q_M(x)}
			\cdot \frac{\sqrt{-\discr \big( X^L_{L+N}(x) \big)}}{2 \pi \mu_L'(x)}.
		\]
	\end{claim}
	For the proof, let us consider a positive continuous function $f: K \rightarrow \RR$. By applying
	the formula that appears at the bottom of the page 363 of \cite{GeronimoVanAssche1991}, we get
	\begin{equation}
		\label{eq:43}
		\begin{aligned}
		&
		\lim_{j \to \infty} 
		\int_\RR f(x) \Big|\alpha_{L+jN}(x) p^L_{L+(j+1)N}(x) - p^L_{L+jN}(x) \Big|^2 \mu_L({\rm d} x) \\
		&\qquad\qquad=
		\frac{1}{\pi a_L^L} 
		\int_\RR f(x) \big| w_{N-1}^{[1]}(x) \big| \sqrt{1 - \abs{T(x)}^2} {\: \rm d} x
		\end{aligned}
	\end{equation}
	where
	\[
		T(x) = \tfrac{1}{2} \tr\big( X^L_{L+N}(x)\big),
	\]
	and for each $m \in \NN_0$, $\big( w_n^{[m]} : n \in \NN_0 \big)$ is the sequence of orthonormal
	polynomials associated with $N$-periodic sequences $\big( a_n^L : n \geq L+m \big)$ and 
	$\big( b_n^L : n \geq L+m \big)$, and $\alpha_{L+jN}(x)$ is a specific solution of the equation
	\begin{equation} \label{eq:58}
		\alpha^2 - 2 T(x) \alpha + 1 = 0.
	\end{equation}
	Observe that
	\[
		\Big|\alpha_{L+jN}(x) p^L_{L+(j+1)N}(x) - p^L_{L+jN}(x) \Big|
		=
		\Big|\overline{\alpha_{L+jN}(x)} p^L_{L+(j+1)N}(x) - p^L_{L+jN}(x) \Big|.
	\]
	Since $\lambda^L(x)$ is a solution of \eqref{eq:58}, we obtain
	\begin{align*}
		\Big|\alpha_{L+jN}(x) p^L_{L+(j+1)N}(x) - p^L_{L+jN}(x) \Big|
		&=
		\Big|\lambda^L(x) p^L_{L+(j+1)N}(x) - p^L_{L+jN}(x) \Big| \\
		&=
		\Big|p^L_{L+(j+1)N}(x) - \overline{\lambda^L(x)} p^L_{L+jN}(x) \Big|.
	\end{align*}
	Hence, by Claim \ref{clm:5}, for $j \geq 1$,
	\[
		\Big|\alpha_{L+jN}(x) p^L_{L+(j+1)N}(x) - p^L_{L+jN}(x) \Big|
		=
		\Big|p^L_{L+2 N}(x) - \overline{\lambda^L(x)} p^L_{L+N}(x) \Big|.
	\]
	Moreover, by \cite[Proposition 3]{PeriodicIII}, we have
	\[
		X_{L+N}^L(x) = 
		\begin{pmatrix}
			-\frac{a^L_{L+N-1}}{a^L_{L}} w^{[1]}_{N-2}(x) & w^{[0]}_{N-1}(x) \\
			-\frac{a^L_{L+N-1}}{a^L_{L}} w^{[1]}_{N-1}(x) & w^{[0]}_{N}(x)
		\end{pmatrix}.
	\]
	Since $\mu_L$ is absolutely continuous on $\Lambda_L$, by \eqref{eq:43}, we obtain
	\begin{equation}
		\label{eq:60}
		\Big|p^L_{L+2N}(x) - \overline{\lambda^L(x)} p^L_{L+N}(x) \Big|^2 \mu_L'(x) 
		=
		\frac{1}{2 \pi a^L_{L+N-1}} 
		\Big| \big[X^L_{L+N}(x)\big]_{2,1} \Big| \cdot \sqrt{-\discr \big( X^L_{L+N}(x) \big)},
	\end{equation}
	for almost all $x \in K$. Since both sides of \eqref{eq:60} are continuous functions, the 
	equality in \eqref{eq:60} is for all $x \in K$. Lastly, let us observe that 
	$\abs{t_j^L} \equiv 1$ for $j \geq k+1$. Therefore, by \eqref{eq:65} and \eqref{eq:62},
	\begin{align*}
		\prod_{j=M+1}^{k+1} \big| t^L_j(x) \big|^2 
		&= \prod_{j = M+1}^k \big| t_j(x) \big|^2 \\
		&= \frac{a_{(M+1)N+i-1}}{a_{L+N-1}} \det Q_k^{-1}(x) \det Q_M(x),
	\end{align*}
	and the claim follows.

	Now, we are in the position to prove the theorem. By Corollary~\ref{cor:9} and Claim \ref{clm:3},
	\[
		\lim_{k \to \infty} \big|\phi_{L_k+N}^{L_k}(x)\big|^2 = 
		\frac{\left|\calX_{21}(x)\right|}{a_{(M+1)N+i-1} \det Q_M(x)}
		\cdot \frac{\sqrt{-\discr \calX(x)}}{2 \pi \nu'(x)},
	\]
	uniformly with respect to $x \in K$. Since
	\[
		\Big\|X^{L_k}_{L_k+N} - X_{L_k+N} \Big\| 
		\leq
		\Big\|X^{L_k}_{L_k+N} - X_{L_k} \Big\| + \Big\|X_{L_k} - X_{L_{k+1}} \Big\|,
	\]
	by Corollary \ref{cor:3}, we have
	\[
		\lim_{k \to \infty} \sup_{x \in K} \Big\|X^{L_k}_{L_k+N}(x) - X_{L_k+N}(x) \Big\| = 0.
	\]
	Therefore, by Claim \ref{clm:4},
	\[
		\lim_{k \to \infty}
		\big| \phi_{L_k+N}(x) \big|^2 
		=
		\frac{\left|\calX_{21}(x)\right|}{a_{(M+1)N+i-1} \det Q_M(x)}
		\cdot \frac{\sqrt{-\discr \calX(x)}}{2 \pi \nu'(x)},
	\]
	uniformly with respect to $x \in K$. Since, by \eqref{eq:62}, 
	\[
		\lim_{k \to \infty} \sup_{x \in K}
		\Big|
		a_{(k+1)N+i-1} \prod_{j = M+1}^{k} \big|t_j(x) \big|^2
		-
		a_{(M+1)N+i-1} \det Q_M(x)
		\Big| = 0,
	\]
	by Proposition \ref{prop:8}, we conclude the proof.
\end{proof}

\begin{corollary} \label{cor:7}
Under the hypotheses of Theorem~\ref{thm:3} one has
\[
	\sqrt{a_{nN+i-1}} p_{nN+i}(x) = 
	\sqrt{\frac{2 \abs{\calX_{21}(x)}}{\pi \nu'(x) \sqrt{-\discr \calX(x)}}}
	\sin \Big( \sum_{j=M+1}^{n} \theta_j(x) + \vphi(x) \Big) + o_K(1), \qquad x \in K
\]
for some continuous functions $\theta_j : K \to (0, \pi)$ satisfying
\begin{equation} \label{eq:144}
	\theta_j(x) = \arccos \Big( \tfrac{1}{2} \tr \calX(x) \Big) + o_K(1).
\end{equation}
\end{corollary}
\begin{proof}
By Corollary~\ref{cor:8} there is constant $C>0$ such that for every $n \geq 1$ and $x \in K$,
\[
	\sqrt{a_{(n+1)N+i-1}} |p_{nN+i}(x)| \leq C.
\]
Since $\discr X_{nN+i}$ is a polynomial of degree at most $2N$,
\[
	\lim_{n \to \infty} \sup_{x \in K} 
	\Big| \sqrt{-\discr X_{nN+i}(x)} - \sqrt{-\discr \calX(x)} \Big| = 0.
\]
Plugging \eqref{eq:139} into \eqref{eq:140}, we obtain
\[
	\sqrt{a_{nN+i-1}} p_{nN+i}(x) = 
	\sqrt{\frac{2 \abs{\calX_{21}(x)}}{\pi \nu'(x) \sqrt{-\discr \calX(x)}}}
	\sin \Big( \sum_{j=M+1}^{n} \arg t_j(x) + \vphi(x) \Big) + o_K(1), \qquad x \in K.
\]
For $j \geq M$, we set
\[
	\theta_j(x) = \arg t_{j}(x).
\]
Then \eqref{eq:143} implies \eqref{eq:144} and the corollary follows.
\end{proof}

\section{Applications} 
\label{sec:2}
In this section we present applications of the main results of this article. To simplify the exposition let us
first introduce some notation. For any positive integers $N$ and $r$, we say that a real sequence $(x_n : n \geq 0)$
belongs to $\calD^N_{r, 0}$ if for every $i \in \{0, 1, \ldots, N-1 \}$,
\[
	(x_{nN+i} : n \geq 0) \in \calD_{r,0}(\RR).
\]
Moreover, we shall use $N$-step difference operator defined by
\[
	\Delta_N x_n = x_{n+N} - x_n.
\]

\begin{proposition} \label{prop:9}
Let $N$ and $r$ be positive integers. Suppose that for $i \in \{0, 1, \ldots, N-1 \}$,
\[
	\bigg( \frac{a_{nN+i-1}}{a_{nN+i}} : n \in \NN \bigg),\ 
	\bigg( \frac{b_{nN+i}}{a_{nN+i}} : n \in \NN \bigg),\ 
	\bigg( \frac{1}{a_{nN+i}} : n \in \NN \bigg) \in \calD_{r,0}.
\]
Then for every compact $K \subset \RR$,
\[
	(B_{nN+i} : n \in \NN) \in \calD_{r,0} \big(K, \GL(2, \RR) \big).
\]
\end{proposition}
\begin{proof}
First of all, for every $X \in \GL(2, \RR)$,
\begin{equation} \label{eq:133}
	\| X \| \leq \| X \|_2 \leq \| X \|_1,
\end{equation}
where $\| X \|_t$ the $t$-norm of the matrix considered as the element of $\RR^{4}$.
For every $j \in \NN_0$,
\[
	\Delta^j B_{nN+i}(x) = 
	\begin{pmatrix}
		0 & \Delta^j (1) \\
		- \Delta^j \Big( \frac{a_{nN+i-1}}{a_{nN+i}} \Big) & 
		x \Delta^j \Big( \frac{1}{a_{nN+i}} \Big) - 
		\Delta^j \Big( \frac{b_{nN+i}}{a_{nN+i}} \Big)
	\end{pmatrix}.
\]
Hence, by \eqref{eq:133} and compactness of $K$ there is a constant $c > 0$ such that
\[
	\sup_{x \in K} \| \Delta^j B_{nN+i}(x) \| \leq 
	\Delta^j(1) +
	\Big| \Delta^j \Big(\frac{a_{nN+i-1}}{a_{nN+i}} \Big) \Big| +
	c \Big| \Delta^j \Big( \frac{1}{a_{nN+i}} \Big) \Big| +
	\Big| \Delta^j \Big( \frac{b_{nN+i}}{a_{nN+i}} \Big) \Big|.
\]
From which the conclusion easily follows.
\end{proof}

\subsection{Asymptotically periodic case}
Let $N$ be a positive integer, and let $(\alpha_n : n \in \ZZ)$ and $(\beta_n : n \in \ZZ)$ be $N$-periodic sequences 
of positive and real numbers, respectively. For any $i \in \{0, 1, \ldots, N-1 \}$, let us define
\begin{equation} 
	\label{eq:114}
	\calX_i(x) =
	\prod_{j=i}^{N+i-1}
	\calB_{j}(x)
	\quad \text{where} \quad
	\calB_j(x) = 
	\begin{pmatrix}
		0 & 1 \\
		-\frac{\alpha_{j-1}}{\alpha_j} & \frac{x - \beta_j}{\alpha_j}
	\end{pmatrix}.
\end{equation}
Let $A_\per$ be the Jacobi matrix on $\ell^2(\NN_0)$ associated with the sequences $\alpha$ and $\beta$. Then
\[
	(\tr \calX_0)^{-1}\big((-2,2)\big) = \bigcup_{j=1}^N I_j
\]
where $I_j$ are open non-empty disjoint intervals. Moreover, 
\[
	\sigmaEss{A_\per} = (\tr \calX_0)^{-1} \big( [-2,2] \big),
\]
and the corresponding measure $\mu_\per$ is purely absolutely continuous on every $I_j$ with positive real analytic
density (see, e.g. \cite[Chapter 5]{Simon2010}).

Since we have a good understanding of $A_\per$, it is natural to consider Jacobi matrices $A$ which are compact
perturbations of $A_\per$, that is
\begin{equation} 
	\label{eq:115}
	\lim_{n \to \infty} |a_n - \alpha_n| = 0, \qquad
	\lim_{n \to \infty} |b_n - \beta_n| = 0.
\end{equation}
Observe that by Weyl's theorem $\sigmaEss{A} = \sigmaEss{A_\per}$. Let us decompose of the corresponding measure 
\[
	\mu({\rm d} x) = \mu_{\textrm{ac}}'(x) \ud x + \mu_{\textrm{s}}({\rm d} x),
\]
where $\mu_{\textrm{s}}$ is a singular measure. In \cite[Theorem 6]{GeronimoVanAssche1991}, it was shown that $\mu$
is purely absolutely continuous on every $I_j$ with positive continuous density provided that
$(a_n), (b_n) \in \calD^N_{1, 0}$. Moreover, in \cite[Theorem 1]{Kaluzny2012}, it was proven that
$\mu_{\textrm{ac}}'(x) > 0$ for almost all $x \in \sigmaEss{A}$ provided that
\[
	\sum_{n=0}^\infty | \Delta_N a_n |^2 + 
	\sum_{n=0}^\infty | \Delta_N b_n |^2  < \infty.
\]
Furthermore, in \cite[Theorem 1.5]{Lukic2018} it was shown that this conclusion might not hold if \eqref{eq:115} is not
satisfied. This result is optimal in the sense that for $a_n \equiv 1$ and $p > 2$ the set of all sequences
$(b_n) \in \ell^p$ such that $\mu$ is purely singular continuous on $[-2, 2]$, is Baire typical 
(see \cite[Theorem 4.4]{Simon1995}). Hence, in order to obtain absolute continuity of $\mu$ for sequences satisfying
$(\Delta_N a_n : n \in \NN_0), (\Delta_N b_n : n \in \NN_0) \in \ell^p$ with $p>2$, one has to assume some additional
hypotheses. For example, in \cite[Theorem 1.2]{Lukic2011}, a sufficient condition was given to guarantee continuity and
positivity of the density of $\mu$ on $(-2,2) \setminus S$ for some finite set $S$. Under different conditions,
namely $(a_n), (b_n) \in \calD^N_{r, 0}$ for some $r \geq 1$, it was shown in \cite[Corollary 5.12]{Moszynski2009} that
$\mu$ is purely absolutely continuous on every $I_j$. The following corollary shows that in this setup one also has
positive continuous density.
\begin{corollary}
	\label{cor:4}
	Let $N$ and $r$ be positive integers. Suppose that the sequences $(a_{n} : n \in \NN_0)$ and $(b_{n} : n \in \NN_0)$
	belong to $\calD^N_{r, 0}$ and satisfy
	\begin{equation}
		\label{eq:14}
		\lim_{n \to \infty} |a_n - \alpha_n| = 0, \qquad\text{and}\qquad
		\lim_{n \to \infty} |b_n - \beta_n| = 0.
	\end{equation}
	Let $\calX_i$ be defined by \eqref{eq:114}, and let $K$ be a compact subset of
	\[
		\Lambda = \big\{ x \in \RR: | \tr \calX_0(x) | < 2 \big\}.
	\]
	Then
	\[
		g(x) = \lim_{n \to \infty} a_{n+N-1} 
		\big| p_{n}(x) p_{n+N-1}(x) - p_{n-1}(x) p_{n+N}(x) \big|, \qquad x \in K,
	\]
	defines a continuous strictly positive function. Moreover, $\mu$ is purely absolutely continuous on 
	$K$ with the density
	\[
		\mu'(x) = \frac{\sqrt{4 - \big(\tr \calX_0(x)\big)^2}}{2 \pi g(x)}, \qquad x \in K.
	\]
	There are $M>0$ and real continuous functions $\{\eta_0, \eta_1, \ldots, \eta_{N-1}\}$ on $K$, 
	such that for all $n > M$ and $i \in \{0, 1, \ldots, N-1\}$,
	\begin{equation} \label{eq:123}
		\sqrt{a_{nN+i-1}} p_{nN+i}(x) = 
		\sqrt{\frac{2 \big|[\calX_i(x)]_{2,1}\big|}{\pi \mu'(x) \sqrt{4 - \big(\tr \calX_0(x)\big)^2}}}
		\sin \Big( \sum_{j=M+1}^{n} \theta_{jN+i}(x) + \eta_i(x) \Big) + o_K(1), \quad x \in K
	\end{equation}
	for some continuous functions $\theta_m: K \rightarrow (0, \pi)$ satisfying
	\[
		\theta_m(x) = \arccos \Big( \tfrac{1}{2} \tr \calX_0(x) \Big) + o_K(1), \qquad x \in K.
	\]
\end{corollary}
\begin{proof}
	Since $(a_n : n \in \NN) \in \calD^N_{r,0}$ and $\inf_{n \geq 0} a_n > 0$, Corollary~\ref{cor:2} implies
	\[
		\Big( \frac{1}{a_n} : n \in \NN \Big) \in \calD^N_{r,0}.
	\]
	Thus, by Corollary~\ref{cor:1},
	\[
		\Big( \frac{a_{n-1}}{a_n} : n \in \NN \Big),\ 
		\Big( \frac{b_n}{a_n} : n \in \NN \Big) \in \calD^N_{r,0}.
	\]
	Let $K \subset \Lambda$ be compact and let $i \in \{0, 1, \ldots, N-1 \}$. By Proposition~\ref{prop:9}, for every 
	$j \in \{0, 1, \ldots, N-1 \}$,
	\[
		\big( B_{nN+j} : n \in \NN \big) \in \calD_{r, 0} \big(K, \GL(2, \RR) \big),
	\]
	which, by Corollary~\ref{cor:1}, implies that
	\[
		\big( X_{nN+i} : n \in \NN \big) \in \calD_{r, 0} \big(K, \GL(2, \RR) \big).
	\]
	By \eqref{eq:14},
	\begin{equation}
		\label{eq:134}
		\lim_{n \to \infty} X_{nN+i} = \calX_i
	\end{equation}
	locally uniformly on $\RR$. Since 
	\[
		\calX_{i} = 
		\big( \calB_{i-1} \ldots  \calB_0 \big) 
		\calX_0 
		\big( \calB_{i-1} \ldots  \calB_0 \big)^{-1}
	\]
	one has $\discr \calX_i = \discr \calX_0$, and consequently, \eqref{eq:134} gives
	\[
		\lim_{n \to \infty} \discr X_{nN+i} = \discr \calX_0 = (\tr \calX_0)^2 - 4.
	\]
	Moreover,
	\[
		\lim_{n \to \infty} \frac{a_{(n+1)N+i-1}}{a_{nN+i-1}} = \frac{\alpha_{i-1}}{\alpha_{i-1}} = 1.
	\]
	The Carleman condition entails that the moment problem for $(p_n : n \in \NN_0)$ is determinate. 
	Hence, by Corollary~\ref{cor:9},
	\[
		g_i(x) = \lim_{\stackrel{n \to \infty}{n \equiv i \bmod N}} 
		a_{n+N-1} \big|p_{n}(x) p_{n+N-1} - p_{n-1}(x) p_{n+N}(x)\big|, \qquad x \in K,
	\]
	defines a continuous and strictly positive function. Moreover,
	\[
		\mu'(x) = \frac{\sqrt{4 - \big( \tr \calX_0(x) \big)^2}}{2 \pi g_i(x)}, \qquad x \in K,
	\]
	and consequently, $g_i = g_0$ for every $i$. Finally, to complete the proof let us observe that the asymptotic
	\eqref{eq:123} is a consequence of Corollary~\ref{cor:7}.
\end{proof}
The asymptotic \eqref{eq:123} for $N=1$ and $r=1$ has been proven in \cite[Theorem 1]{Totik1985}. 
Later, the extension to $N > 1$ has been achieved in \cite[formula (3.27) and Theorem 5]{GeronimoVanAssche1991}. 
Here, the case $r > 1$ is a new result.

\subsection{Periodic modulations}
Let $N$ be a positive integer, and let $(\alpha_n : n \in \NN_0)$ and $(\beta_n : n \in \NN_0)$ be $N$-periodic sequences 
of positive and real numbers, respectively. Let $\calX_i$ be defined in \eqref{eq:114}. If the sequences 
$(a_n : n \in \NN_0)$ and $(b_n : n \in \NN_0)$ satisfy
\begin{equation} 
	\label{eq:117}
	\lim_{n \to \infty} a_n = \infty, \qquad
	\lim_{n \to \infty} \bigg| \frac{a_{n-1}}{a_n} - \frac{\alpha_{n-1}}{\alpha_n} \bigg| = 0, \qquad\text{and}\qquad
	\lim_{n \to \infty} \bigg| \frac{b_n}{a_n} - \frac{\beta_n}{\alpha_n} \bigg| = 0,
\end{equation}
then $A$ is called a Jacobi matrix with periodically modulated entries. A special case of this class has been studied
in \cite{JanasNaboko2002}, namely sequences satisfying $\sum_{n=0}^\infty 1/a_n = \infty$ and
\[
	a_n = \alpha_n \tilde{a}_n, \qquad
	b_n = \beta_n \tilde{a}_n
\]
for some sequences $(\tilde{a}_{n-1}/\tilde{a}_{n} : n \in \NN_0)$ and $(1/\tilde{a}_n : n \in \NN_0)$ belonging to
$\calD_{1, 0}$. There it was shown that the measure $\mu$ is purely absolutely continuous on $\RR$ provided that
$|\tr \calX_0(0)| < 2$ (see \cite[Theorem 3.1]{JanasNaboko2002}), and purely discrete one when $|\tr \calX_0(0)| > 2$
(see \cite[Theorem 4.2]{JanasNaboko2002}). Afterwards, in \cite[Theorem A]{PeriodicI}, these
results in the case $|\tr \calX_0(0)| < 2$ have been extended to sequences satisfying \eqref{eq:117} provided that
$(a_{n-1}/a_{n} : n \in \NN_0)$, $(b_n/a_n : n \in \NN_0)$ and $(1/a_n : n \in \NN_0)$ belong to $\calD^N_{1, 0}$. 
In \cite[Theorem 1]{PeriodicII} it was shown that under the same hypothesis the measure $\mu$ has continuous positive
density. In \cite[Theorem D]{PeriodicIII} some similar results were obtained in the case $|\tr \calX_0(0)| = 2$.

The following corollary is an extension of \cite[Theorem 1]{PeriodicII} to the general $r \geq 1$.
\begin{corollary}
	\label{cor:5}
	Let $N$ and $r$ be positive integers. Suppose that sequences $(a_{n-1}/a_n : n \geq 1)$, $(b_n/a_n : n \geq 0)$ and
	$(1/a_n : n \geq 0)$ belong to $\calD^N_{r, 0}$ and satisfy \eqref{eq:117}. Let $\calX_i(x)$ be defined in 
	\eqref{eq:114} and let $K$ be a compact subset of $\RR$. If $|\tr \calX_0(0)| < 2$, then for every 
	$i \in \{0, 1, \ldots, N-1 \}$,
	\begin{equation} \label{eq:145}
		g_i(x) 
		= \lim_{\stackrel{n \to \infty}{n \equiv i \bmod N}}
		a_{n+N-1} \big|p_{n}(x) p_{n+N-1}(x) - p_{n-1}(x) p_{n+N}(x)\big|, \qquad x \in K,
	\end{equation}
	defines a continuous strictly positive function. Moreover, the measure $\nu_i$ with the 
	density defined by
	\begin{equation} \label{eq:146}
		\nu_i'(x) = \frac{\sqrt{4 - \big(\tr \calX_0(0)\big)^2}}{2 \pi g_i(x)}, \qquad x \in \RR,
	\end{equation}
	is an orthonormalizing measure for the polynomials $(p_n : n \in \NN_0)$. Moreover, there
	are $M > 0$ and real continuous functions $\{\eta_0, \eta_1, \ldots, \eta_{N-1}\}$ on $K$, such that for all $n > M$
	and $i \in \{0, 1, \ldots, N-1\}$,
	\begin{equation} \label{eq:124}
		\sqrt{a_{nN+i-1}} p_{nN+i}(x) = 
		\sqrt{\frac{2 \big| [\calX_i(0)]_{2,1}\big|}
		{\pi \nu_i'(x) \sqrt{4 - \big(\tr \calX_0(x)\big)^2}}}
		\sin \Big( \sum_{j=M+1}^{n} \theta_{jN+i}(x) + \eta_i(x) \Big) + o_K(1), \quad x \in K
	\end{equation}
	for some continuous functions $\theta_m: K \rightarrow (0, \pi)$ satisfying
	\[
		\theta_m(x) = \arccos \Big( \tfrac{1}{2} \tr \calX_0(0) \Big) + o_K(1), \qquad x \in K.
	\]
\end{corollary}
\begin{proof}
Let $K \subset \Lambda$ be compact and let $i \in \{0, 1, \ldots, N-1 \}$.
Then Proposition~\ref{prop:9} implies that for every $j \in \{0, 1, \ldots, N-1 \}$,
\[
	\big( B_{nN+j} : n \in \NN \big) \in \calD_{r, 0} \big(K, \GL(2, \RR) \big),
\]
and, by Corollary~\ref{cor:1}, we obtain
\[
	\big( X_{nN+i} : n \in \NN \big) \in \calD_{r, 0} \big(K, \GL(2, \RR) \big).
\]
By \eqref{eq:117},
\begin{equation}
	\label{eq:19}
	\lim_{n \to \infty} X_{nN+i}(x)= \calX_i(0)
\end{equation}
locally uniformly on $\RR$. Since 
\[
	\calX_{i} = 
	\big( \calB_{i-1} \ldots  \calB_0 \big) 
	\calX_0 
	\big( \calB_{i-1} \ldots  \calB_0 \big)^{-1}
\]
one has $\discr \calX_i = \discr \calX_0$, and consequently, \eqref{eq:19} implies
\[
	\lim_{n \to \infty} \discr X_{nN+i}(x) = \discr \calX_0(0) = \big( \tr \calX_0(0) \big)^2 - 4.
\]
Moreover, by \eqref{eq:117}
\[
	\lim_{n \to \infty} 
	\bigg| \prod_{j=0}^{N-1} \frac{a_{n+j-1}}{a_{n+j}} - 
	\prod_{j=0}^{N-1} \frac{\alpha_{n+j-1}}{\alpha_{n+j}} \bigg| = 0.
\]
Since the products are telescoping, we obtain
\[
	\lim_{n \to \infty} \bigg| \frac{a_{n-1}}{a_{n+N-1}} - 1 \bigg| = 0.
\]
Hence, Corollary~\ref{cor:9} implies the existence of the limit \eqref{eq:145} and the 
formula \eqref{eq:146}. Finally, the asymptotic \eqref{eq:124} follows from Corollary~\ref{cor:8}. 
The proof is complete.
\end{proof}

Observe that if $\sum_{n=0}^\infty 1/a_n = \infty$, then $g_i = g_0$ for every 
$i \in \{0, 1, \ldots, N-1\}$. We expect that it is always the case without an additional hypothesis.

Let us discuss the case $N=1$ and $\sum_{n=0}^\infty 1/a_n = \infty$. In \cite[Theorem 3.1]{JanasNaboko2001} it was proven that the measure $\nu$ is absolutely continuous provided $r \geq 1$.
Moreover, in \cite[Theorem 3]{AptekarevGeronimo2016} the asymptotic \eqref{eq:124} was proven under the assumption $r=1$. For $N>1$ the result is new even for $r=1$.

\subsection{A blend of bounded and unbounded parameters}
Let $N$ be a positive integer, and let $(\alpha_n : n \in \NN_0)$ and $(\beta_n : n \in \NN_0)$ be 
$N$-periodic sequences of positive and real numbers, respectively. Suppose that positive sequences 
$\tilde{a}$ and $\tilde{c}$, and a real sequence $\tilde{b}$, satisfy
\begin{equation} 
\label{eq:118}
	\lim_{n \to \infty} \big|\tilde{a}_n - \alpha_n\big| = 0, \qquad 
	\lim_{n \to \infty} \big|\tilde{b}_n - \beta_n\big| = 0, \qquad\text{and}\qquad
	\lim_{n \to \infty} \tilde{c}_n = \infty.
\end{equation}
For $k \geq 0$ and $i \in \{0, 1, \ldots, N+1 \}$, we define
\begin{equation} 
	\label{eq:119}
	\begin{aligned}
	a_{k(N+2) + i} &=
	\begin{cases}
		\tilde{a}_{kN+i} & \text{if } i \in \{ 0, 1, \ldots, N-1 \}, \\
		\tilde{c}_{2k+i-N} & \text{if } i \in \{ N, N+1 \},
	\end{cases} \\
	b_{k(N+2) + i} &= 
	\begin{cases}
		\tilde{b}_{kN+i} & \text{if } i \in \{ 0, 1, \ldots, N-1 \}, \\
		0 & \text{if } i \in \{ N, N+1 \}.
	\end{cases}
	\end{aligned}
\end{equation}
The sequences of the form \eqref{eq:119} were considered in \cite[Theorem 5]{Janas2011} where it was assumed that 
\[
	(\tilde{a}_n - \alpha_n : n \in \NN_0) \in \ell^2 \cap \calD_{1,0}^N, \qquad 
	\tilde{b}_n = 0, \qquad \tilde{c}_{2k+i} = (k+1)^\tau + \gamma_{i, k},
\]
for some $\tau \in \big(\tfrac{1}{2}, 1\big)$, and
\[
	\bigg( \frac{\gamma_{i, k}}{(k+1)^\tau} : k \in \NN_0 \bigg) \in \ell^2 \cap \calD_{1,0}, \qquad i\in\{0,1\}.
\]
Under the above hypotheses it was shown that the measure $\mu$ is purely absolutely continuous on $\Lambda$
and $\sigmaEss{A} = \overline{\Lambda}$. The following corollary additionally implies that its density is continuous and
positive, and provides an asymptotic information on some subsequences of the orthogonal polynomials.
\begin{corollary} 
\label{cor:6}
	Let $N$ and $r$ be positive integers. Suppose that \eqref{eq:118} is satisfied, together with
	\[
		\lim_{k \to \infty} \frac{\tilde{c}_{2k+1}}{\tilde{c}_{2k}} = 1,
	\]
	\[
	\big( \tilde{a}_n : n \in \NN_0 \big),\ \big( \tilde{b}_n : n \in \NN_0 \big) \in \calD^N_{r,0},
	\]
	and
	\[
	\bigg( \frac{1}{\tilde{c}_{2n} \tilde{c}_{2n+1}} : n \in \NN_0 \bigg),\ 
	\bigg( \frac{\tilde{c}_{2n+1}}{\tilde{c}_{2n}} : n \in \NN_0 \bigg) \in \calD_{r,0}.
	\]
Let the sequences $(a_n : n \in \NN_0)$ and $(b_n : n \in \NN_0)$ be defined in \eqref{eq:119}. For
$i \in \{1, 2, \ldots, N \}$, we set
\[
	\calX_i(x) =
	\Bigg\{
	\prod_{j=1}^{i-1} 
	\calB_j(x)
	\Bigg\}
	\calC(x)
	\Bigg\{
	\prod_{j=i}^{N-1} 
	\calB_j(x)
	\Bigg\}
\]
where 
\[
	\calB_j(x) =
	\begin{pmatrix}
		0 & 1 \\
		-\frac{\alpha_{j-1}}{\alpha_j} & \frac{x-\beta_j}{\alpha_j}
	\end{pmatrix},
	\qquad\text{and}\qquad
	\calC(x) = 	
	\begin{pmatrix}
		0 & -1 \\
		\frac{\alpha_{N-1}}{\alpha_0} & -\frac{2x - \beta_0}{\alpha_0}
	\end{pmatrix}.
\]
Let $K$ be a compact subset of
\[
	\Lambda = \big\{ x \in \RR : |\tr \calX_1(x)| < 2 \big\}.
\]
Then for every $i \in \{1, 2, \ldots, N \}$,
\begin{equation}
	\label{eq:120}
	g_i(x) = \lim_{\stackrel{n \to \infty}{n \equiv i \bmod (N+2)}} 
	a_{n+N+1} \big| p_{n}(x) p_{n+N-1}(x) - p_{n-1}(x) p_{n+N}(x) \big|, \qquad x \in K,
\end{equation}
defines a continuous strictly positive function. Moreover, the measure $\mu$ is purely absolutely
continuous on $K$ with the density
\begin{equation}
	\label{eq:121}
	\mu'(x) = \frac{\sqrt{4 - \big(\tr \calX_1(x)\big)^2}}{2 \pi g_i(x)}, \qquad x \in K.
\end{equation}
There are $M>0$ and real continuous functions $\{\eta_1, \eta_2, \ldots, \eta_{N}\}$ on $K$, such that for all
$n > M$ and $i \in \{1, 2, \ldots, N\}$,
\begin{equation} 
	\label{eq:150}
	\begin{aligned}
	\sqrt{a_{n(N+2)+i-1}} p_{n(N+2)+i}(x) 
	&= \sqrt{\frac{2 \big| [\calX_i(x)]_{2,1} \big|}
	{\pi \mu'(x) \sqrt{4 - \big(\tr \calX_0(x)\big)^2}}}
	\sin \Big( \sum_{j=M+1}^{n} \theta_{j(N+2)+i}(x) + \eta_i(x) \Big) \\
	&\phantom{=}+ o_K(1), \qquad x \in K,
	\end{aligned}
\end{equation}
for some continuous functions $\theta_m: K \rightarrow (0, \pi)$ satisfying
\[
	\theta_m(x) = \arccos \Big( \tfrac{1}{2} \tr \calX_1(x) \Big) + o_K(1), \qquad x \in K.
\]
\end{corollary}
\begin{proof}
Let $K \subset \Lambda$ be compact and let $i \in \{1, \ldots, N \}$.
Define
\[
	\tilde{B}_n(x) =
	\begin{pmatrix}
		0 & 1 \\
		-\frac{\tilde{a}_{n-1}}{\tilde{a}_{n}} & \frac{x - \tilde{b}_n}{\tilde{a}_n}
	\end{pmatrix}.
\]
Then the proof of Corollary~\ref{cor:4} implies that for every $j \in \{1, \ldots, N-1\}$,
\begin{equation} \label{eq:147}
	(\tilde{B}_{nN+j} : n \in \NN) \in \calD_{r, 0} \big( K, \GL(2, \RR) \big).
\end{equation}
Since $B_{n(N+2)+j} = \tilde{B}_{nN+j}$,
\begin{equation} \label{eq:148}
	X_{n(N+2)+i} = 
	\prod_{j=i}^{N+2+i-1} B_{n(N+2)+j} = 
	\Bigg\{
	\prod_{j=1}^{i-1} \tilde{B}_{(n+1)N+j}
	\Bigg\}
	C_n
	\Bigg\{
	\prod_{j=i}^{N-1} \tilde{B}_{nN+j}
	\Bigg\}
\end{equation}
where
\[
	C_n = B_{(n+1)(N+2)} B_{n(N+2)+N+1} B_{n(N+2)+N}.
\]
A direct computation shows that
\begin{align*}
	C_n(x) &=
	\begin{pmatrix}
		0 & -\frac{\tilde{c}_{2n}}{\tilde{c}_{2n+1}} \\
		\frac{\tilde{a}_{nN+N-1}}{\tilde{a}_{nN+N}} \frac{\tilde{c}_{2n+1}}{\tilde{c}_{2n}} & 
		-\frac{x-\tilde{b}_{nN+N}}{\tilde{a}_{nN+N}} \frac{\tilde{c}_{2n}}{\tilde{c}_{2n+1}}
		-\frac{x}{\tilde{a}_{nN+N}} \frac{\tilde{c}_{2n+1}}{\tilde{c}_{2n}}
	\end{pmatrix} \\
	&\phantom{=}+
	\frac{x}{\tilde{c}_{2n} \tilde{c}_{2n+1}}
	\begin{pmatrix}
		-\tilde{a}_{nN+N-1} & x \\
		\tilde{a}_{nN+N-1} \frac{x - \tilde{b}_{nN+N}}{\tilde{a}_{nN+N}} &
		x \frac{x-\tilde{b}_{nN+N}}{\tilde{a}_{nN+N}}
	\end{pmatrix}.
\end{align*}
Therefore, by Corollary~\ref{cor:1}, 
\begin{equation} \label{eq:149}
	(C_n : n \in \NN) \in \calD_{r, 0} \big( K, \GL(2, \RR) \big).
\end{equation}
Moreover,
\begin{equation} \label{eq:151}
	\lim_{n \to \infty} \sup_{x \in K} \big\| C_{n}(x) - \calC(x) \big\| = 0.
\end{equation}
Now, Corollary~\ref{cor:1} together with \eqref{eq:147}, \eqref{eq:148} and \eqref{eq:149} imply
\[
	(X_{n(N+2)+i} : n \in \NN) \in \calD_{r, 0} \big( K, \GL(2, \RR) \big).
\]
Since for $j \in \{1, \ldots, N-1 \}$,
\[
	\lim_{n \to \infty} \sup_{x \in K} \big\| \tilde{B}_{nN+j}(x) - \calB_j(x) \big\| = 0,
\]
by \eqref{eq:148} and \eqref{eq:151} one gets
\begin{equation} \label{eq:122}
	\lim_{n \to \infty} \sup_{x \in K} \big\| X_{n(N+2)+i}(x) - \calX_i(x) \big\| = 0.
\end{equation}
Let us observe that 
\[
	\prod_{j=1}^{N-1} \calB_j(x) = 
	\bigg\{ \prod_{j=i}^{N-1} \calB_j(x) \bigg\} 
	\bigg\{ \prod_{j=1}^{i-1} \calB_j(x) \bigg\},
\]
thus using $\tr (AB) = \tr(BA)$, we conclude that for every $i \in \{1, 2, \ldots, N \}$,
\[
	\discr \big( \calX_i(x) \big) = 
	\discr \big( \calX_1(x) \big).
\]
Hence, by \eqref{eq:122},
\[
	\lim_{n \to \infty} \discr X_{n(N+2)+i}(x) = 
	\discr \calX_1(x) = \big(\tr \calX_1(x)\big)^2 - 4.
\]
Since
\[
	\lim_{n \to \infty}
	\frac{a_{(n+1)(N+2)+i-1}}{a_{n(N+2)+i-1}} 
	=
	\lim_{n \to \infty}
	\frac{\tilde{a}_{(n+1)N+i-1}}{\tilde{a}_{nN+i-1}} = 1,
\]
Corollary~\ref{cor:9} implies the existence of the limit \eqref{eq:120}.
The Carleman condition gives that the moment problem for $(p_n : n \in \NN_0)$ is determinate. 
Hence, by Corollary~\ref{cor:9} we obtain \eqref{eq:121} and consequently, $g_i = g_1$. Finally, 
the asymptotic \eqref{eq:150} follows from Corollary~\ref{cor:7}. This completes the proof.
\end{proof}

\begin{bibliography}{jacobi}
	\bibliographystyle{amsplain}
\end{bibliography}

\end{document}